%% file: sym_rev.tex
\documentclass[reqno,11pt]{article}
\usepackage{graphicx, amsmath, amsthm}
\usepackage{amssymb}  
\usepackage{dsfont} 
\input{sym.sty}
\input{graph.sty}

\def\Pip{\smash{P_i^{(p)}}}
\def\Pjp{\smash{P_j^{(p)}}}
\def\uip#1{{u_{#1}^{(p)}}}
\def\hlinespace{\vrule height 11pt depth 4pt width 0pt}
\def\matrixspace{\vrule height 9pt depth 9pt width 0pt}

\def\PA {P_i}
\def\PB {P_j}

\renewcommand{\sitep}{\textbf{\textcolor[rgb]{1,0,0}{{\small$+$}}}}
\renewcommand{\sitem}{\textbf{\textcolor[rgb]{0,0,1}{{\small$-$}}}}
\renewcommand{\sitea}{\textbf{\textcolor[rgb]{1,0,0}{{\small$\alpha$}}}}
\renewcommand{\siteb}{\textbf{\textcolor[rgb]{0,0,1}{{\small$\beta$}}}}

\begin{document}

%%%%%%%%%%%%%%%%%%%%%%%%%%%%%%%%%%%%%%%%%%%%%%%%%%%%%%%%%%%%%%%%%%%%%%%%%%%%%%

\title{The Eyring--Kramers law for \\ Markovian jump processes with symmetries}
\author{Nils Berglund, S\'ebastien Dutercq}
\date{}   

\maketitle

\begin{abstract}
\noindent
We prove an Eyring--Kramers law for the small eigenvalues and mean
first-passage times of a metastable Markovian jump process which is invariant
under a group of symmetries. Our results show that the usual Eyring--Kramers law
for asymmetric processes has to be corrected by a factor computable in
terms of stabilisers of group orbits. Furthermore, the symmetry can produce
additional Arrhenius exponents and modify the spectral gap. The results are
based on representation theory of finite groups. 
\end{abstract}

\leftline{\small{\it Date.\/} December 3, 2013. Revised. October 3, 2014.}
\leftline{\small 2010 {\it Mathematical Subject Classification.\/} 
60J27    %Markov chains with continuous parameter
(primary), 
20C35,   %Applications of group representations to physics
60K35    %Interacting random processes; statistical mechanics type models;
%percolation theory
(secondary)}

%%%%%%%%%%%%%%%%%%%%%%%%%%%%%%%%%%%%%%%%%%%%%%%%%%%%%%%%%%%%%%%%%%%%%%%%%%%%%%

\section{Introduction}
\label{sec_intro}

The Eyring--Kramers law characterises the mean transition times between local
minima of a diffusion in a potential landscape. 
Consider the stochastic differential equation 
\begin{equation}
 \label{meta03}
 \6x_t = - \nabla V(x_t)\6t + \sqrt{2\eps} \6W_t\;,
\end{equation} 
where $V:\R^d\to\R$ is a confining potential, and $W_t$ is a $d$-dimensional
standard Brownian motion. If $V$ has just two quadratic local minima $x^*$ and
$y^*$, separated by a quadratic saddle $z^*$, the Eyring--Kramers law states
that the expected first-passage time $\tau$ from $x^*$ to a small ball around
$y^*$ is given by 
\begin{equation}
 \label{meta05}
  \expecin{x^*}{\tau} = \frac{2\pi}{\abs{\lambda_1(z^\star)}}
\sqrt{\frac{\abs{\det(\hessian V(z^\star))}}{\det(\hessian V(x^\star))}}
\e^{[V(z^\star)-V(x^\star)]/\eps} [1+\Order{\eps^{1/2}\abs{\log\eps}^{3/2}}]\;.
\end{equation} 
Here $\hessian V(x^\star)$ and $\hessian V(z^\star)$ denote the Hessian
matrices of $V$ at $x^\star$ and $z^\star$ respectively, and
$\lambda_1(z^\star)$ is the unique negative eigenvalue of $\hessian V(z^\star)$
($z^\star$ is called a saddle of Morse index~$1$). A critical point is called
quadratic if the Hessian matrix is non-singular. 

The exponential behaviour in $\e^{[V(z^\star)-V(x^\star)]/\eps}$
of~\eqref{meta05} was first proposed by van t'Hoff, and justified physically by
Arrhenius~\cite{Arrhenius}. The more precise formula with the prefactors
depending on Hessians was introduced by Eyring~\cite{Eyring} and
Kramers~\cite{Kramers}. Mathematical proofs for these formulas are much more
recent. The Arrhenius law has first been justified by Wentzell and Freidlin,
using the theory of large deviations~\cite{VF69,VF70}. The first rigorous proof
of the full Eyring--Kramers law~\eqref{meta05} was provided by Bovier, Eckhoff,
Gayrard and Klein, using potential-theoretic methods~\cite{BEGK,BGK}. These
methods have also been applied to lattice models~\cite{BEGK_MC,denHollander04}
and used to extend the Eyring--Kramers law to systems with non-quadratic
saddles~\cite{BG2010}
and to stochastic partial differential equations~\cite{BG12a,Barret_2012}. 
Other approaches to proving~\eqref{meta05} include an analysis of
Witten Laplacians acting on
$p$-forms~\cite{HelfferKleinNier04,LePeutrecNierViterbo12}. See for
instance~\cite{Berglund_irs_MPRF} for a recent survey. 

If the potential $V$ has $N>2$ local minima, the characterisation of metastable
timescales becomes more involved. It has been known for a long time that the
diffusion's generator admits $N$ exponentially small eigenvalues, and that they
are connected to  metastable
timescales~\cite{Mathieu1995,Miclo1995,KolokoltsovMakarov1996,Kolokoltsov00}.
A method introduced by Wentzell, cf~\cite{Ventcel_1972} and \cite[Chapter
6]{FW}, based on so-called $W$-graphs, provides a way to compute the
logarithmic asymptotics of these eigenvalues. That is, Wentzell's algorithm
yields Arrhenius exponents $H_k$ such that the $k$-th small eigenvalue
$\lambda_k$ behaves like $-\e^{-H_k/\eps}$, in the sense that
$\eps\log(-\lambda_k)$ converges to $-H_k$ as $\eps\to0$. In fact, the $W$-graph
algorithm does not require the drift term to be in gradient form, as
in~\eqref{meta03}.

However, for the gradient system~\eqref{meta03}, much more precise
estimates are available. Indeed, results in~\cite{BEGK,BGK} provide expressions
of the form $\lambda_k = -C_k(\eps)\e^{-H_k/\eps}$, where $C_k(0)$ is known
explicitly. Furthermore, one can order the local minima $x^*_1, \dots, x^*_N$ of
$V$ in such a way that the mean transition time from each $x^*_k$ to the set
$\set{x^*_1,\dots,x^*_{k-1}}$ of its predecessors is close to the inverse of
$\lambda_k$. 

The only limitation of these results is that they require a non-degeneracy
condition to hold. In short, all relevant saddle heights have to be different
(see Section~\ref{ssec_asym} for a precise formulation). While this condition
holds for generic potentials $V$, it will fail whenever the potential is
invariant under a symmetry group $G$. Let us mention two examples of such
potentials, which will serve as illustrations of the theory throughout this
work.

\begin{example}
\label{ex:Glauber} 
The papers~\cite{BFG06a,BFG06b} introduce a model with $N$ particles on the
periodic lattice $\Lambda=\Z_N:=\Z/N\Z$, which are coupled harmonically to their
nearest neighbours, and subjected to a local double-well potential
$U(y)=\frac14y^4 - \frac12y^2$. The associated potential reads 
\begin{equation}
 \label{meta04} 
 V_\gamma(x) = \sum_{i\in\Lambda} U(x_i) + \frac{\gamma}{4}
\sum_{i\in\Lambda}(x_{i+1}-x_i)^2\;.
\end{equation} 
The potential $V_\gamma$ is invariant under the group $G$ generated by three
transformations: the cyclic permutation $r:x\mapsto(x_2,\dots,x_N,x_1)$, the
reflection $s:x\mapsto(x_N,\dots,x_1)$ and the sign change $c:x\mapsto-x$. The
transformations $r$ and $s$ generate the dihedral group $D_N$ of isometries
preserving a regular $N$-gon, and since $c$ commutes with $r$ and $s$, $G$ is
the direct product $D_N\times \Z_2$. 

It has been shown in~\cite{BFG06a} that for weak coupling $\gamma$, the model
behaves like an Ising model with Glauber spin-flip
dynamics~\cite[Section~3]{denHollander04}, while for large $\gamma$ the systems
synchronizes, meaning that all components $x_i$ tend to be equal. 
\end{example}

\begin{example}
\label{ex:Kawasaki} 
Consider a variant of the previous model, obtained by restricting $V_\gamma$ to
the space $\setsuch{x\in\R^\Lambda}{\sum x_i=0}$. For weak coupling, this
system will mimic a Kawasaki-type dynamics with conserved \lq\lq particle\rq\rq\
number~\cite[Section~4]{denHollander04}. The symmetry group $G$ is the same as
in the previous example. 
\end{example}

The potential-theoretic approach has been extended to some particular 
degenerate situations, by computing equivalent capacities for systems of
capacitors in series or in parallel~\cite[Chapter 1.2]{Barret_thesis}. This is
close in spirit to the electric-network analogy for random
walks on graphs~\cite{Doyle_Snell}. However, for more complicated symmetric
potentials admitting many local minima, the computation of equivalent
capacities becomes untractable. This is why we develop in this work a general
approach based on Frobenius' representation theory for finite groups. The basic
idea is that each irreducible representation of the group $G$ will yield a
subset of the generator's eigenvalues. The trivial representation corresponds
to initial distributions which are invariant under $G$, while all other
representations are associated with non-invariant distributions. 

In the present work, we concentrate on the case where the process itself is a
Markovian jump process, with states given by the local minima, and transition
probabilities governed by the Eyring--Kramers law between neighbouring minima. 
The study of jump processes is also of independent interest: see
for instance~\cite{Hwang_Scheu_92,Chiang_Chow_93,Trouve_1996} for applications
to simulated annealing and image processing.
We expect that the results can be extended to diffusions of the
form~\eqref{meta01}, is a similar way as in the asymmetric case. The main
results are Theorems~\ref{thm_dim1_trivial}, \ref{thm_dim1_nontrivial}
and~\ref{thm_dimd} in Section~\ref{sec_res}, which provide sharp estimates for
the eigenvalues and relate them to mean transition times.
More precisely, we show that the generator's eigenvalues are of the form 
\begin{equation}
 \label{eq:meta04b}
 \lambda_k = -C_k \e^{-H_k/\eps} \brak{1+\Order{\e^{-\theta/\eps}}}
\end{equation} 
for some $\theta>0$, with explicit expressions of the $C_k$ and $H_k$ in terms
of group orbits and stabilisers. 
While the $H_k$ can also be obtained by algorithms based on $W$-graphs, our
approach provides in addition the prefactors $C_k$ at a comparable computational
cost. The results show in particular that a phenomenon of clustering of
eigenvalues takes place: there are in general many eigenvalues sharing the same
Arrhenius exponent $H_k$, but having possibly different prefactors $C_k$. The
precise determination of the $C_k$ is thus important in order to be able to
distinguish eigenvalues in a same cluster. We also provide a probabilistic
interpretation for this clustering.

The remainder of the article is organised as follows. In Section~\ref{sec_set},
we define the main objects, recall results from the asymmetric case, as well as
some elements of representation theory of finite groups. Section~\ref{sec_res}
contains the results on eigenvalues and transition times for processes that are
invariant under a group of symmetries. These results are illustrated in
Section~\ref{sec_ex} for two cases
of Example~\ref{ex:Kawasaki}. The remaining sections contain the proofs of these
results. Section~\ref{sec_proofG} collects all proofs related to representation
theory, Section~\ref{sec_proofT} contains the estimates of eigenvalues, and
Section~\ref{sec_hit} establishes the links with mean transition times. Finally
Appendix~\ref{appendix} gives some information on the computation of the
potential landscape of Example~\ref{ex:Kawasaki}.  

\medskip

{\it Notations:} We denote by $a\wedge b$ the minimum of two real numbers $a$
and $b$, and by $a\vee b$ their maximum. If $A$ is a finite set, $\abs{A}$
denotes its cardinality, and $\indicator{A}(x)$ denotes the indicator function
of $x\in A$. We write $\one$ for the identity matrix, and $\vone$
for the constant vector with all components equal to $1$. The results concern
Markovian jump processes $\set{X_t}_{t\geqs0}$ on finite sets $\cX$, defined on
a filtered probability space $(\Omega,\cF,\fP,\set{\cF_t}_{t\geqs0})$. We denote
their generator by $L$, that is, $L$ is a matrix with non-negative off-diagonal
elements, and zero row sums. The law of $X_t$ starting with an initial
distribution $\mu$ is denoted $\probin{\mu}{\cdot}$, and
$\expecin{\mu}{\cdot}$ stands for associated expectations. If $\mu=\delta_i$ is
concentrated in a single point, we write $\probin{i}{\cdot}$ and
$\expecin{i}{\cdot}$.

%\newpage

%%%%%%%%%%%%%%%%%%%%%%%%%%%%%%%%%%%%%%%%%%%%%%%%%%%%%%%%%%%%%%%%%%%%%%%%%%%%%%

\section{Setting}
\label{sec_set}

%%%%%%%%%%%%%%%%%%%%%%%%%%%%%%%%%%%%%%%%%%%%%%%%%%%%%%%%%%%%%%%%%%%%%%%%%%%%%%

\subsection{Metastable markovian jump processes}
\label{ssec_meta}

Let $\cX$ be a finite set, and let $L$ be the generator of an irreducible 
Markovian jump process on $\cX$. We assume that the elements of $L$ can be
written in the form 
\begin{equation}
 \label{meta01}
 L_{ij} = \frac{c_{ij}}{m_i} \e^{-h_{ij}/\eps} \;, 
 \qquad
 i, j\in\cX\;, i\neq j\;,
\end{equation} 
where $\eps>0$, $c_{ij} = c_{ji} > 0$, $m_i > 0$ and $0< h_{ij} \leqs
+\infty$
(it will be convenient to write $h_{ij}=+\infty$ to indicate that $L_{ij}=0$). 
In addition, we assume that there exists a function $V:\cX\to \R_+$ such that 
$L$ is reversible with respect to the measure $m\e^{-V/\eps}$:
\begin{equation}
 \label{meta02}
 m_i\e^{-V_i/\eps}L_{ij} = m_j\e^{-V_j/\eps}L_{ji}
 \qquad
 \forall i, j\in\cX\;.
\end{equation} 
Since we assume $c_{ij} = c_{ji}$, this is equivalent to   
\begin{equation}
 \label{meta02a}
 V_i + h_{ij} = V_j + h_{ji} 
 \qquad
 \forall i, j\in\cX\;.
\end{equation} 
Our aim is to understand the behaviour as $\eps\to0$ of the Markov process
$X_t$ of generator $L$, when $L$ is invariant under a group $G$ of bijections 
$g:\cX\to\cX$. 

Let $\cG=(\cX,E)$ be the undirected graph with set of edges
$E=\setsuch{(i,j)\in\cX^2}{L_{ij}>0}$. It will be convenient to associate with
an edge $e=(i,j)\in E$ the \defwd{height of the saddle between $i$ and $j$}
defined by $V_e=V_i + h_{ij} = V_j + h_{ji}$, and to write 
$c_e = c_{ij} = c_{ji}$.
In particular, this convention justifies the graphical representation used
e.g.\ in \figref{fig:hierarchy} below, in which we draw a function $V(x)$ with
local minima at height $V_i$ and saddles at height $V_e$.

A widely used method for determining the logarithmic asymptotics of the
eigenvalues $\lambda_k$ of the generator $L$ relies on so-called
$W$-graphs~\cite{Ventcel_1972}. Given a subset $W\subset\cX$, a $W$-graph is a
directed graph with set of vertices $\cX$, such that every point
$i\in\cX\setminus W$ is the origin of exactly one path ending in a point $j\in
W$. Then one has 
\begin{equation}
 \label{eq:Arrhenius}
 H_k = -\lim_{\eps\to0} \eps\log(-\lambda_k) = V^{(k)} - V^{(k+1)}\;, 
\end{equation} 
where each $V^{(k)}$ involves a minimum over all $W$-graphs with $k$ elements. 
The quantity to be minimized is the sum of $h_{ij}$ over all edges of the
graph. This is the continuous-time analogue of~\cite[Theorem~7.3]{FW}, see for
instance~\cite[Section~4.3]{Cameron_Vanden-Eijnden_2014}. However, our aim is
to determine the Eyring--Kramers prefactor of the eigenvalues as well, that is,
to obtain the constants  
\begin{equation}
 \label{eq:EyringKramers}
 C_k = \lim_{\eps\to 0} (-\lambda_k) \e^{H_k/\eps}\;.
\end{equation} 
In the \lq\lq general position\rq\rq\ case, i.e.\ when the relevant potential
differences are all different, the $C_k$ can be determined from the
graph $\cG$ and the notions of communication height and metastable
hierarchy~\cite{BEGK,BGK}. We explain this approach in the next two subsections,
before turning to the more difficult case where the generator $L$ is invariant
under a symmetry group.

%%%%%%%%%%%%%%%%%%%%%%%%%%%%%%%%%%%%%%%%%%%%%%%%%%%%%%%%%%%%%%%%%%%%%%%%%%%%%%

\subsection{Metastable hierarchy}
\label{ssec_asym}

\begin{figure}[t]
%\vspace{-15mm}
\begin{center}
\scalebox{1.0}{
{\raise 25mm \hbox{\input{graph_example1}}}
}
\hspace{14mm}
\scalebox{0.8}{
\input{potential_communication}
}
\end{center}
\vspace{-3mm}
\caption[]{The graph $\cG$ and potential associated with
Example~\ref{ex:threewell}. The communication height $H(2,1)$ from state 
$2$ to state $1$ is given by $h_{231} = V_{(1,3)} - V_2$.}
\label{fig:hierarchy} 
\end{figure}
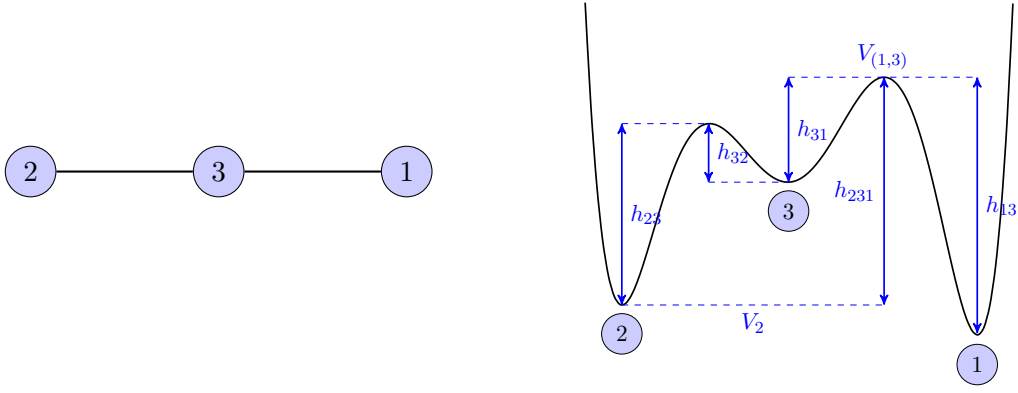

\begin{definition}[Communication heights]
\label{def_comm_heights} 
Let $i\neq j\in\cX$. For $p\geqs 1$, 
the \defwd{$(p+1)$-step communication height} from $i$ to $j$ is defined
inductively by
\begin{equation}
 \label{eq:asym01} 
h_{ik_1\dots k_pj} = h_{ik_1\dots k_p} \vee
(h_{ik_1}-h_{k_1i}+h_{k_1k_2}-h_{k_2k_1}+\dots+h_{k_pj})
\end{equation}
(see~\figref{fig:hierarchy}).
The \defwd{communication height} from $i$ to $j\neq i$ is defined by 
\begin{equation}
 \label{eq:asym02} 
 H(i,j) = \min_{\gamma:i\to j} h_\gamma\;,
\end{equation} 
where the minimum runs over all paths $\gamma=(i,k_1,\dots,k_p,j)$ of length
$p+1\geqs1$. Any such path realising the minimum in~\eqref{eq:asym02} is called
a \defwd{minimal path from $i$ to $j$}. 
If $i\not\in A\subset\cX$, we define the communication height from
$i$ to $A$ as  
\begin{equation}
 \label{eq:asym03} 
 H(i,A) = \min_{j\in A} H(i,j)\;.
\end{equation} 
\end{definition}

Using~\eqref{meta02a} and induction on $p$, it is straightforward to show that 
communication heights can be equivalently defined in terms of heights of
saddles, by 
\begin{equation}
 \label{eq:asym03a}
 h_{ik_1\dots k_pj} + V_i = V_{(i,k_1)} \vee V_{(k_1,k_2)} 
 \vee \dots \vee V_{(k_p,j)} \;.
\end{equation} 
Thus $H(i,j) + V_i$ is the minimum over all paths $\gamma$ from $i$ to $j$ of
the maximal saddle height encountered along $\gamma$. 

\begin{example}
\label{ex:threewell}
Consider the generator 
\begin{equation}
 \label{eq:asym03aa}
 L = 
 \begin{pmatrix}
 \matrixspace
  - \frac{c_{13}}{m_1}\e^{-h_{13}/\eps} 
  & 0 & \frac{c_{13}}{m_1}\e^{-h_{13}/\eps} \\
  \matrixspace
  0 & -\frac{c_{23}}{m_2}\e^{-h_{23}/\eps}
  & \frac{c_{23}}{m_2}\e^{-h_{23}/\eps} \\
  \matrixspace
  \frac{c_{13}}{m_3}\e^{-h_{31}/\eps} 
  & \frac{c_{23}}{m_3}\e^{-h_{32}/\eps} 
  & - \frac{c_{13}}{m_3}\e^{-h_{31}/\eps} - \frac{c_{23}}{m_3}\e^{-h_{32}/\eps}
 \end{pmatrix}
\end{equation}  
where we assume $h_{32} < h_{31}$, $h_{32} < h_{23}$ and $h_{23}-h_{32}+h_{31} <
h_{13}$. The associated graph and potential are shown in~\figref{fig:hierarchy}.
We have for instance 
\begin{equation}
 \label{eq:asym03b}
 h_{231} = h_{23} \vee (h_{23}-h_{32}+h_{31}) = h_{23}-h_{32}+h_{31} = 
 V_{(1,3)} - V_2\;,
\end{equation} 
which is the height difference between state $2$ and the saddle connecting
states $1$ and $3$. The communication height from $2$ to $1$ is given by
\begin{equation}
 \label{eq:asym03c}
 H(2,1) = h_{21} \wedge h_{231} = h_{231}\;,
\end{equation} 
because $h_{21}=+\infty$, and paths of length larger than $3$ have a larger
cost, as they contain several copies of at least one edge. Proceeding
similarly for the other communication
heights, we obtain that the matrix of $H(i,j)$ is given by 
\begin{equation}
 \label{eq:asym03d}
 \begin{pmatrix}
   * & h_{13} & h_{13} \\
   h_{231} & * & h_{23} \\
   h_{31} & h_{32} & *
 \end{pmatrix}\;.
\end{equation}
\end{example}

The following non-degeneracy assumption is the central condition for being in
what we will call the \emph{asymmetric case}. It is a weak form of similar
\lq\lq general position\rq\rq\ assumptions, found e.g.
in~\cite[Section~6.7]{FW}, \cite[p.~225]{Hwang_Scheu_92}
or~\cite[Assumption~1]{Cameron_2014a}.

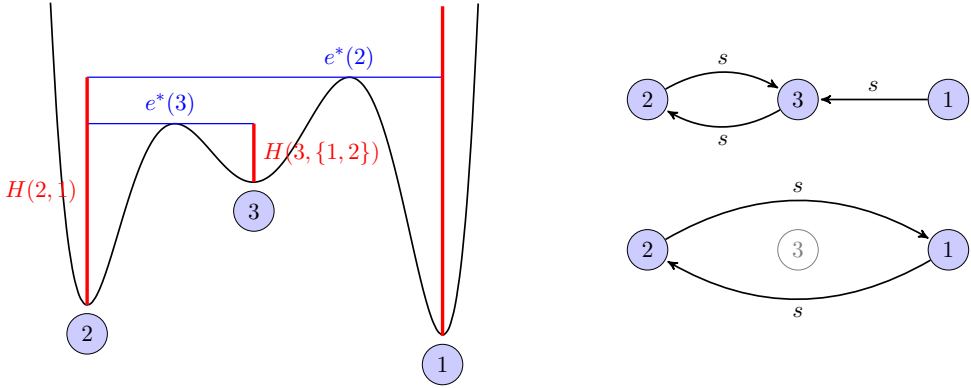
\begin{figure}[t]
%\vspace{-15mm}
\begin{center}
\scalebox{0.8}{
\input{potential_disconnect_graph}
}
\hspace{14mm}
\scalebox{0.8}{
{\raise 10mm \hbox{\input{graph_successors_example1}}}
}
\end{center}
\vspace{-3mm}
\caption[]{Disconnectivity tree and metastable hierarchy for
Example~\ref{ex:threewell}. State $3$ is at the top of a $2$-cycle in the graph
of successors, and is removed after one step. State $2$ is at the top of a
$2$-cycle of the resulting graph of successors. The Arrhenius exponents of
eigenvalues are the lengths of thick vertical segments. The rightmost segment is
considered as infinitely long and corresponds to the eigenvalue $0$. The graphs
of successors are shown for the original system, and for the system in which
state $3$ has been removed.}
\label{fig:metastable_hierarchy} 
\end{figure}

\begin{assump}[Metastable hierarchy]
\label{assump_asym_meta} 
The elements of $\cX=\set{1,\dots,n}$ can be ordered in such a way that 
if $\cM_k=\set{1,\dots,k}$, 
\begin{equation}
 \label{eq:asym04}
 H(k,\cM_{k-1}) \leqs 
 \min_{i<k} H(i, \cM_k\setminus\set{i}) - \theta\;, 
 \qquad k=2,\dots,n
\end{equation} 
for some $\theta>0$. We say that the order $1\prec2\prec\dots\prec n$ defines
the \defwd{metastable hierarchy} of $\cX$
(see~\figref{fig:metastable_hierarchy}).
Furthermore, for each $k$ there is a unique edge $e^*(k)$ such that any minimal
path $\gamma: k\to\cM_{k-1}$ reaches height $H(k,\cM_{k-1}) + V_k$ only on the
edge $e^*(k)$. Finally, any non-minimal path $\gamma: k\to\cM_{k-1}$ reaches at
least height $V_{e^*(k)} + \theta$.
\end{assump}

Condition~\eqref{eq:asym04} means that in the lower-triangular part of the
matrix of communication heights $H(i,j)$, the minimum of each row is smaller, by
at least $\theta$, than the minimum of the row above. Such an ordering will
typically only exist if $L$ admits no nontrivial symmetry group. 

\begin{example}
Returning to the previous example, we see that $H(3,\set{1,2}) = h_{32}$ is
smaller than both $H(1,\set{2,3}) = h_{13}$ and $H(2,\set{1,3}) = h_{23}$.
Furthermore, $H(2,\set{1}) = h_{231}$ is smaller than $H(1,\set{2}) = h_{13}$.
Thus the system admits a metastable order given by $1 \prec 2 \prec 3$, where
$\theta$ is the minimum of $h_{13}-h_{32}$, $h_{23}-h_{32}$ and $h_{13} -
h_{213}$ (\figref{fig:metastable_hierarchy}). The associated highest edges are
$e^*(3)=(2,3)$ and $e^*(2)=(1,3)$.

Note that the relevant communication heights $H(k,\cM_{k-1})$ are
indeed given by the minima of the corresponding row in the subdiagonal part of
the matrix~\eqref{eq:asym03d}. See the next section for algorithms
allowing to determine the metastable order in an effective way. 
\end{example}

The following result is essentially equivalent to~\cite[Theorem~1.2]{BGK}, but
we will provide a new proof that will be needed for the symmetric case. 

\begin{theorem}[Asymptotic behaviour of eigenvalues]
\label{thm_asym} 
If Assumption~\ref{assump_asym_meta} holds,
then for sufficiently small $\eps$, the eigenvalues of $L$ are given by 
$\lambda_1=0$ and 
\begin{equation}
 \label{eq:asym05}
 \lambda_k = -\frac{c_{e^*(k)}}{m_k} \e^{-H(k,\cM_{k-1})/\eps}
 \bigbrak{1+\Order{\e^{-\theta/\eps}}}\;, 
 \qquad
 k=2,\dots,n\;.
\end{equation} 
Furthermore, let $\tau_{\cM_{k-1}} = \inf\setsuch{t>0}{X_t\in\cM_{k-1}}$ be
the first-hitting time of $\cM_{k-1}$. Then for
$k=2,\dots,n$,
\begin{equation}
 \label{eq:asym05b}
 \expecin{i}{\tau_{\cM_{k-1}}} =
\frac{1}{\abs{\lambda_k}}\bigbrak{1+\Order{\e^{-\theta/\eps}}}
\end{equation} 
holds for all initial values $i\in\cX\setminus\cM_{k-1}$.  
\end{theorem}

\begin{example}
The theorem shows that the eigenvalues of the generator~\eqref{eq:asym03aa}
satisfy 
\begin{align*}
\lambda_1 &= 0\;, \\
\lambda_2 &= -\frac{c_{13}}{m_2}
\e^{-h_{231}/\eps}\bigbrak{1+\Order{\e^{-\theta/\eps}}}\;, \\
\lambda_3 &= -\frac{c_{23}}{m_3}
\e^{-h_{32}/\eps}\bigbrak{1+\Order{\e^{-\theta/\eps}}}\;. 
\end{align*}
This can of course be checked by an explicit computation in this
simple case. The interest of Theorem~\ref{thm_asym} is that is works for
arbitrarily large systems, at a relatively modest computational cost.
\end{example}

%%%%%%%%%%%%%%%%%%%%%%%%%%%%%%%%%%%%%%%%%%%%%%%%%%%%%%%%%%%%%%%%%%%%%%%%%%%%%%

\subsection{Computational cost of the algorithm}
\label{ssec_asymcost}

Let us denote by $n=\abs{\cX}$ the number of states, and by $m=\abs{E}$ the
number of edges of $\cG$.
%and by $\hat m$ the maximal incidence number of states (i.e.
%the maximum over $i\in\cX$ of the number of edges sharing state $i$). 
Since we assume the chain to be irreducible, we obviously have $n-1 \leqs m
\leqs \frac12 n(n-1)$. Note that the number of
possible $W$-graphs on $\cX$ is at least $2^n$ (the number of subsets of $\cX$),
so that applying \eqref{eq:Arrhenius} directly to compute the exponents $H_k$
can be very time-consuming. However, in the reversible case the method can be
substantially improved.

\begin{figure}[t]
\begin{center}
\scalebox{0.8}{
\hbox{\input{potential_fill_graph}}
}
\hspace{14mm}
\scalebox{0.75}{
\input{graph_algorithm_example1}
}
\end{center}
\vspace{-5mm}
\caption[]{The operation of removing one state from the system can be viewed
as filling its potential well, up to the height of the lowest reachable saddle.
On the directed graph with weights $h_{ij}$, this amounts to subtracting
$h_{i,s(i)}$ from all edges starting in the state $i$ to be removed, and
replacing any $h_{jk}$ for $j,k\neq i$ by $h_{jk}\wedge(h_{ji}+h_{ik})$.}
\label{fig:algorithm} 
\end{figure}
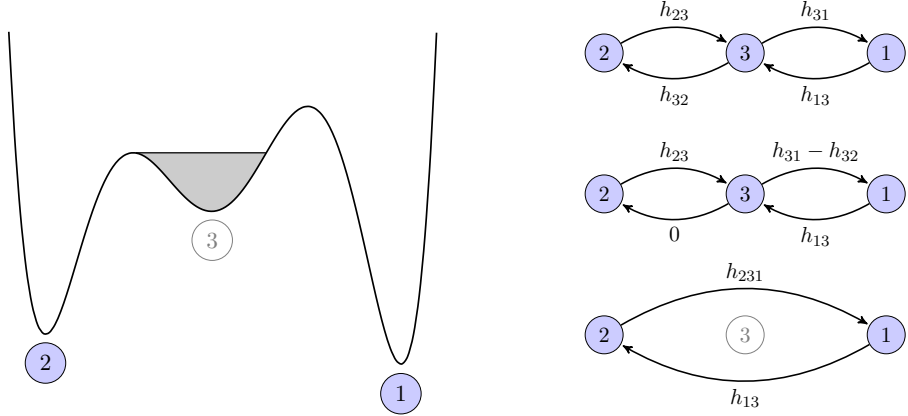

The basic steps of an algorithm determining the metastable hierarchy, and hence
both the prefactors $C_k$ and exponents $H_k$ of the eigenvalues, are as
follows~\cite[Section~4.3]{Cameron_Vanden-Eijnden_2014}:
\begin{itemiz}
\item 	Find the smallest ($1$-step) communication height $h_{ij}$. The state
$i$ will then be the last in the metastable hierarchy. 
\item 	Remove state $i$ and all edges containing $i$ from the graph;
recompute the $1$-step communication heights, in the sense that for each
$j,k\neq i$, $h_{jk}$ is replaced by its minimum with $h_{jik}$
(cf.\ \figref{fig:algorithm}).
\item 	Repeat until there are no edges left. 
\end{itemiz}

One way of doing this efficiently goes as follows. For any site $i\in\cX$, we
call \emph{successor} of $i$ any site $s(i)$ such that
\begin{equation}
 \label{eq:cos01}
 \inf_{j\neq i} h_{ij} = h_{is(i)}\;.
\end{equation} 
If a metastable hierarchy exists, then the last site in the hierarchy has a
unique successor. Define the graph of successors to be the oriented graph with
set of vertices $\cX$ and edges $i\to s(i)$
(\figref{fig:metastable_hierarchy}). Reversibility implies that this
graph cannot have cycles of length larger than $2$. Furthermore, there has to
be at least one cycle of length $2$. If $(i,j)$ is a $2$-cycle and $V_i < V_j$,
we say that $i$ is at the bottom of the cycle and $j$ is at the top of the
cycle. Our algorithm reads 
\begin{itemiz}
\item 	Determine the graph of successors. 
\item 	Find all cycles of length $2$. Determine the top of each cycle. 
Erase all sites which are at the top of a cycle and the corresponding edge, and
update the $h_{jk}$ as above.  
\item 	Repeat until there are no more edges.
\end{itemiz}
This algorithm yields the so-called \emph{disconnectivity tree}, which encodes
the metastable hierarchy. The leaves of the tree are the states in $\cX$. Two
branches join whenever they belong to a $2$-cycle at some step of the algorithm.
By plotting each leaf $i$ at height $V_i$ and joining top and bottom of a cycle
at height $V_{(i,s(i))}$, the Arrhenius exponents and the $e^*(k)$ determining
the prefactors $C_k$ can be read off the tree
(\figref{fig:metastable_hierarchy}).

Jacobi's eigenvalue algorithm allows to diagonalise a general symmetric matrix
in $\Order{n^3}$ steps~\cite{Rutishauser66}. The computational cost of the above
algorithm, by contrast, is between $\Order{m}$ and $\Order{n^2}$ at most.
Indeed, the graph of successors can be determined by finding, for each
$i\in\cX$, the minimum of the $h_{ij}$, which requires $2m$ steps. The cost of
determining all $2$-cycles is of order $n$ and thus negligible. Finally,
updating communication heights and graph of successors after removing one site
has a cost $\Order{n}$ at most, and has to be done for $n-1$ sites in total.
Since $2m<n^2$, the claim follows. 

An alternative algorithm for determining the metastable hierarchy is discussed
in~\cite{Cameron_2014a}. Though it has been derived in order to determine the
exponents $H_k$ only, Theorem~\ref{thm_asym} shows that it can be used to
compute the prefactors $C_k$ as well. The algorithm consists in first computing
the minimal spanning tree of the graph (minimal in terms of communication
heights), and then removing edges from the spanning tree. The computational cost
given in~\cite{Cameron_2014a} is at least $\Order{n\log n}$ and at most
$\Order{n^2}$. Hence the costs of both algorithms discussed here have comparable
bounds, though in specific situations one or the other algorithm may perform
substantially better.

\goodbreak

%%%%%%%%%%%%%%%%%%%%%%%%%%%%%%%%%%%%%%%%%%%%%%%%%%%%%%%%%%%%%%%%%%%%%%%%%%%%%%

\subsection{Symmetry groups and their representations}
\label{ssec_sym}

Let $G$ be a finite group of bijections $g:\cX\to\cX$. We denote by $\pi(g)$
the permutation matrix 
\begin{equation}
 \label{eq:sym01}
 \pi(g)_{ab} = 
 \begin{cases}
 1 & \text{if $g(a)=b$\;,} \\
 0 & \text{otherwise\;.}
 \end{cases}
\end{equation} 
From now on we assume that the generator $L$ is \defwd{invariant} under $G$,
that is, 
\begin{equation}
 \label{eq:sym02}
 \pi(g) L = L \pi(g) 
 \quad
 \forall g\in G\;.
\end{equation} 
This is equivalent to assuming $L_{ab} = L_{g(a)g(b)}$ for all $a,b\in\cX$ and
all $g\in G$. 

Let us recall a few definitions from basic group theory. 

\begin{definition} \hfill
\begin{enum}
\item 	For $a\in\cX$, $O_a=\setsuch{g(a)}{g\in G} \subset \cX$ is called the
\defwd{orbit} of $a$. 
\item 	For $a\in\cX$, $G_a=\setsuch{g\in G}{g(a)=a} \subset G$ is called the
\defwd{stabiliser} of $a$.
\item 	For $g\in G$, $\cX^g=\setsuch{a\in\cX}{g(a)=a} \subset \cX$ is called
the \defwd{fixed-point set} of $g$. 
\end{enum}
\end{definition}

The following facts are well known:

\begin{itemiz}
\item 	The orbits form a partition of $\cX$, denoted $\cX/G$. 
\item 	For any $a\in\cX$, the stabiliser $G_a$ is a subgroup of $G$.
\item 	For any $a\in \cX$, the map
$\varphi:gG_a\mapsto g(a)$ provides a bijection from the set $G/G_a$ of left
cosets to the orbit $O_a$ of $a$, and thus $\abs{G}/\abs{G_a}=\abs{O_a}$. 
\item 	For any $g\in G$ and any $a\in\cX$, one has $G_{g(a)} = gG_ag^{-1}$,
i.e.\ stabilisers of a given orbit are conjugated.
\item 	\defwd{Burnside's lemma:} 
$\sum_{g\in G}\abs{\cX^g} = \abs{G}\abs{\cX/G}$. 
\end{itemiz}

We will denote the orbits of $G$ by $A_1, \dots, A_{n_G}$. The value of
the communication height $H(a,A_j)$ is the same for all $a\in A_i$, and we will
denote it $H(A_i,A_j)$. Similarly, we write $V_{A_i}$ for the common value of
all $V_a$, $a\in A_i$.
We shall make the following two non-degeneracy
assumptions:

\begin{assump}[Metastable order of orbits]
\label{assump_sym_meta} 
Let $\cM_k=A_1\cup\dots\cup A_k$. 
One can order the orbits in such a way that 
\begin{equation}
 \label{eq:sym03}
 H(A_k,\cM_{k-1}) \leqs 
 \min_{i<k} H(A_i, \cM_k\setminus A_i) - \theta\;, 
 \qquad k=2,\dots,n_G
\end{equation} 
for some $\theta>0$. We indicate this by writing $A_1\prec A_2\prec\dots\prec
A_{n_G}$. Furthermore, for each $k=2,\dots,n_G$, there is an edge $e^*(k)\in E$
such
that 
\begin{equation}
 \label{eq:sym03a}
 H(A_k,\cM_{k-1}) + V_{A_k} = V_{(a,b)}
 \quad\Leftrightarrow\quad
 \exists g\in G \colon (g(a),g(b)) = e^*(k)\;.
\end{equation} 
\end{assump}

\begin{assump}[Absence of accidental degeneracy]
\label{assump_sym_nondeg} 
Whenever there are elements $a_1$, $b_1$, $a_2$, $b_2\in \cX$ such that
$h_{a_1b_1}=h_{a_2b_2}$, there exists $g\in G$ such that
$g(\set{a_1,b_1})=\set{a_2,b_2}$. 
\end{assump}

We make the rather strong Assumption~\ref{assump_sym_nondeg} mainly to simplify
the expressions for eigenvalues; the approach we develop here can be applied
without this assumption, at the cost of more complicated expressions. 
Assumption~\ref{assump_sym_nondeg} implies the following property of the matrix
elements of $L$:

\begin{lemma}
\label{lem_GaGb} 
For all $a, b\in\cX$, belonging to different orbits, 
$L_{ah(b)} = L_{ab}$ if and only if $h\in G_aG_b$. 
\end{lemma}
\begin{proof}
By Assumption~\ref{assump_sym_nondeg}, and since $g(a)\neq b$ as they belong to 
different orbits, $L_{ah(b)} = L_{ab}$ if and only if there is a $g\in G$
such that $g(a)=a$ and $g(b)=h(b)$. This is equivalent to the existence of a
$g\in G_a$ such that $g(b)=h(b)$, i.e.\ $b=g^{-1}h(b)$. This in turn is
equivalent to $h\in G_aG_b$.
\end{proof}

Direct transitions between two orbits $A_i$ and $A_j$ are dominated by those
edges $(a,b)$ for which $h_{ab}$ is minimal. We denote the minimal value 
\begin{equation}
 \label{eq:sym04}
 h^*(A_i,A_j) = \inf\setsuch{h_{ab}}{a\in A_i,b\in A_j}\;.
\end{equation} 
Note that $h^*(A_i,A_j)$ may be infinite (if there is no edge between the
orbits), and that $H(A_i,A_j)\leqs h^*(A_i,A_j)$. By decreasing, if necessary,
the value of $\theta>0$, we may assume that 
\begin{equation}
 \label{eq:sym04a}
 h_{ab} > h^*(A_i,A_j),\; a\in A_i,\; b\in A_j 
 \quad\Rightarrow\quad
 h_{ab} \geqs h^*(A_i,A_j) + \theta\;.
\end{equation} 

By Lemma~\ref{lem_GaGb}, if
$h^*(A_i,A_j)$ is finite, each $a\in A_i$ is connected to exactly
$\abs{G_aG_b}/\abs{G_b}$ states in $A_j\ni b$ with transition rate $L_{ab} =
[c_{ab}/m_a]\e^{-h^*(A_i,A_j)/\eps}$. Observe that
\begin{align}
\nonumber
 \varphi: G_aG_b/G_b & \to G_a/(G_a\cap G_b) \\
  gG_b &\mapsto g(G_a\cap G_b)
 \label{eq:sym04b}
\end{align} 
is a bijection, and therefore the number $n^a_j$ of states in $A_j$
communicating with $a$ can be written in either of the two equivalent forms   
\begin{equation}
 \label{eq:sym04c}
 n^a_j = 
\frac{\abs{G_aG_b}}{\abs{G_b}} = \frac{\abs{G_a}}{\abs{G_a\cap G_b}}\;. 
\end{equation}

The map $\pi$ defined by~\eqref{eq:sym01} is a morphism from $G$ to
$\GL(n,\C)$, and thus defines a \defwd{representation} of $G$ (of dimension
$\dim\pi = n$). 
In what follows, we will draw on some facts from representation theory of finite
groups (see for instance~\cite{Serre_groups}):
\begin{itemiz}
\item	 A representation of $G$ is
called \defwd{irreducible} if there is no proper subspace of $\C^n$ which is
invariant under all $\pi(g)$. 
\item 	Two representations $\pi$ and $\pi'$ of
dimension $d$ of $G$ are called \defwd{equivalent} if there exists a matrix 
$S\in\GL(d,\C)$ such that $S\pi(g)S^{-1}=\pi'(g)$ for all $g\in G$. 
\item 	Any finite group $G$ has only finitely many inequivalent irreducible
representations $\pi^{(0)}, \dots, \pi^{(r-1)}$. Here
$\pi^{(0)}$ denotes the \defwd{trivial representation}, $\pi^{(0)}(g)=1$
$\forall g\in G$. 
\item 	Any representation $\pi$ of $G$ can be decomposed into irreducible
representations:
\begin{equation}
 \label{eq:sym05}
 \pi = \bigoplus_{p=0}^{r-1} \alpha^{(p)} \pi^{(p)}\;, \qquad
 \alpha^{(p)} \geqs 0\;, \quad
 \sum_{p=0}^{r-1} \alpha^{(p)} \dim(\pi^{(p)}) = \dim(\pi) = n\;.
\end{equation} 
This means that we can find a matrix $S\in\GL(n,\C)$ such that all matrices
$S\pi(g)S^{-1}$ are block diagonal, with $\alpha^{(p)}$ blocks given by
$\pi^{(p)}(g)$. This decomposition is unique up to equivalence and the order of
factors. 
\item 	
For any irreducible representation $\pi^{(p)}$ contained in $\pi$, 
let $\chi^{(p)}(g) = \Tr \pi^{(p)}(g)$ denote its \defwd{characters}. Then 
\begin{equation}
 \label{eq:sym06}
 P^{(p)} = \frac{\dim (\pi^{(p)})}{\abs{G}} \sum_{g\in G} \cc{\chi^{(p)}(g)}
\pi(g) 
\end{equation} 
is the projector on the invariant subspace of $\C^n$ associated with
$\pi^{(p)}$. In particular, 
\begin{equation}
 \label{eq:sym07}
 \alpha^{(p)} \dim (\pi^{(p)}) = \Tr P^{(p)} = 
 \frac{\dim (\pi^{(p)})}{\abs{G}} \sum_{g\in G} \cc{\chi^{(p)}(g)} \chi(g)\;,
\end{equation} 
where $\chi(g) = \Tr \pi(g)$. Note that for the representation defined
by~\eqref{eq:sym01}, we have $\chi(g) = \abs{\cX^g}$. 
\end{itemiz}

\begin{example}[Irreducible representations of the dihedral group]
\label{ex:dihedral} 
The dihedral group $D_N$ is the group of symmetries of a regular $N$-gon. It is
generated by $r$, the rotation by $2\pi/N$, and $s$, one of the reflections
preserving the $N$-gon. In fact 
\begin{equation}
D_N=\bigset{\id,r,r^2,\dots,r^{N-1},s,rs,r^s,\dots,r^{N-1}s} 
\end{equation} 
is entirely specified by the conditions $r^N=\id$, $s^2=\id$ and $rs=sr^{-1}$. 
If $N$ is even, then $D_N$ has $4$ irreducible representations of dimension
$1$, specified by $\pi(r)=\pm1$ and $\pi(s)=\pm1$. In addition, it has
$\frac{N}{2}-1$ irreducible representations of dimension $2$, equivalent to 
\begin{equation}
 \label{eq:Dn_2dim}
 \pi(r) = 
 \begin{pmatrix}
 \e^{2\icx\pi k/N} & 0 \\ 0 & \e^{-2\icx\pi k/N}
 \end{pmatrix}\;,
 \qquad
 \pi(s) = 
 \begin{pmatrix}
 0 & 1 \\ 1 & 0
 \end{pmatrix}\;, 
 \qquad
 k=1,\dots\frac{N}{2}-1\;.
\end{equation} 
The associated characters are given by 
\begin{equation}
 \label{eq:Dn_characters}
 \chi(r^is^j) = \Tr\pi(r^is^j) 
 = 2 \cos \biggpar{\frac{2\pi ik}{N}} \delta_{j0}\;,
 \qquad
 i=0,\dots,N-1,\; j=0,1\;.
\end{equation} 
There are no irreducible representations of dimension larger than $2$. 
If $N$ is odd, there are $2$ irreducible representations of dimension $1$,
specified by $\pi(r)=1$ and $\pi(s)=\pm1$, and
$(N-1)/2$ irreducible representations of dimension $2$. 
\end{example}

%%%%%%%%%%%%%%%%%%%%%%%%%%%%%%%%%%%%%%%%%%%%%%%%%%%%%%%%%%%%%%%%%%%%%%%%%%%%%%

\section{Results}
\label{sec_res}

The central idea of our approach is to use the decomposition~\eqref{eq:sym05} of
the permutation matrix representation~\eqref{eq:sym01} of the symmetry group $G$
to characterise the eigenvalues of $L$. It follows from~\eqref{eq:sym02}
and~\eqref{eq:sym06} that 
\begin{equation}
 \label{eq:res01}
 P^{(p)}L = LP^{(p)}\;, 
 \qquad
 p=0, \dots, r-1\;, 
\end{equation} 
so that the $r$ images $P^{(p)}\C^n$ (where $n=\abs{\cX}$) are invariant
subspaces for $L$. We can thus determine the eigenvalues of $L$ by restricting
the analysis to each restriction $L^{(p)}$ of $L$ to the subspace
$P^{(p)}\C^n$. All eigenvalues of $L$ are associated with exactly one
irreducible representation, so that the procedure will eventually yield the
whole spectrum of $L$. 

An equivalent way of stating this is that the projectors $P^{(p)}$ will allow
us to construct a basis in which $L$ is block-diagonal. Each block corresponds
to a different irreducible representation. After computing the matrix elements
of each block, the eigenvalues can be determined by adapting the algorithm
of the asymmetric case.  

\subsection{The trivial representation}

Let us start by the restriction $L^{(0)}$ of $L$ to the subspace $P^{(0)}\C^n$
associated with the trivial representation $\pi^{(0)}$.

\begin{prop}[Matrix elements of $L^{(0)}$ for the trivial 
representation]
\label{prop:L_trivial} 
The subspace $P^{(0)}\C^n$ has dimension $n_G$ and is spanned by the vectors
$u_i^{(0)} = \indicator{A_i}$, $i=1,\dots n_G$. 
The off-diagonal matrix element of $L^{(0)}$ for
transitions between the orbits $A_i$
and $A_j$ is given by 
\begin{equation}
 \label{eq:sym09}
 L^{(0)}_{ij} := 
 \frac{\pscal{u_i}{Lu_j}}{\pscal{u_i}{u_i}} = 
 \frac{c^*_{ij}}{m^*_i}
\e^{-h^*(A_i,A_j)/\eps} [1+\Order{\e^{-\theta/\eps}}]\;,
\end{equation} 
with  
\begin{equation}
 \label{eq:sym10}
 c^*_{ij} = \frac{c_{ab}}{\abs{G_{a}\cap G_{b}}}\;,
 \qquad
 {m^*_i} = \frac{m_a}{\abs{G_a}}\;,
\end{equation} 
where $a\in A_i$ and $b\in A_j$ are such that $h_{ab}=h^*(A_i,A_j)$.
Furthermore, $L^{(0)}$ is a generator, and thus its diagonal elements are given
by 
\begin{equation}
 \label{eq:sym10b}
 L^{(0)}_{ii} =: - \sum_{j\neq i} L^{(0)}_{ij}\;.
\end{equation} 
\end{prop}

The basis vectors $u_i$ are indicator functions on the orbits $A_i$. Thus 
if the initial distribution $\mu$ is uniform on each $A_i$, then it stays
uniform on each $A_i$ for all times. The process $X_t$ is then equivalent to the
process on $\set{1,\dots,n_G}$ with transition rates given by $L^{(0)}$.
Applying Theorem~\ref{thm_asym} on the asymmetric case to this process, which
is possible thanks to Assumption~\ref{assump_sym_meta}, we thus
obtain the following Kramers formula for the eigenvalues of $L^{(0)}$.

\begin{theorem}[Eigenvalues associated with the trivial representation]
\label{thm_dim1_trivial} 
If Assumptions \ref{assump_sym_meta} and~\ref{assump_sym_nondeg} hold true,
then for $\eps$ small enough, the spectrum of $L^{(0)}$ consists in 
$n_G$ eigenvalues of geometric multiplicity $1$, given by $\lambda^{(0)}_1=0$
and 
\begin{equation}
 \label{eq:sym08}
 \lambda^{(0)}_k = 
 \frac{c^*_{i(k)j(k)}}{m^*_k} \e^{-H(A_k,\cM_{k-1})/\eps}
 \bigbrak{1+\Order{\e^{-\theta/\eps}}}\;, 
 \qquad
 k=2,\dots,n_G\;,
\end{equation}
where $i(k)$ and $j(k)$ are such $e^*(k)=(a,b)$ with $a\in A_{i(k)}$ and
$b\in A_{j(k)}$ (cf~\eqref{eq:sym03a}). 
Furthermore, for $2\leqs k\leqs n_G$, let $\mu$ be a probability distribution
supported on $\cX\setminus\cM_{k-1}$ which is uniform on each $A_j$. Then 
\begin{equation}
 \label{eq:sym08a}
  \expecin{\mu}{\tau_{\cM_{k-1}}} =
\frac{1}{\abs{\lambda_k}}\bigbrak{1+\Order{\e^{-\theta/\eps}}}\;.
\end{equation} 
\end{theorem}

The main difference between the Kramers formula~\eqref{eq:sym08} of the
symmetric case and its equivalent~\eqref{eq:asym05} for the asymmetric case is
that the eigenvalues are multiplied by an extra factor $\abs{G_c}/\abs{G_a\cap
G_b}$, where $c\in A_k$, $a\in A_{i(k)}$ and $b\in A_{j(k)}$, which accounts for
the symmetry. 

\subsection{Other irreducible representations of dimension $1$}

Theorem~\ref{thm_dim1_trivial} only accounts for a small subset of $n_G$
eigenvalues of the generator, associated with distributions that are uniform on
each orbit $A_i$. All other eigenvalues of $L$ will be associated to the rate
at which non-uniform initial distributions approach the uniform one. We first
determine eigenvalues associated with nontrivial representations of dimension
$1$, which are easier to obtain. The following lemma shows that given such a
representation, only part of the orbits may be present in the image of the
associated projector. 

\begin{lemma}
\label{lem_dim1} 
Let $\pi^{(p)}$ be an irreducible representation of dimension $1$ of $G$, 
let $A_i$ be an orbit of $G$ and fix any $a\in A_i$. Denote by $\pi_i(g)$ the
permutation induced by $g\in G$ on $A_i$ and let $\Pip$ be the associated
projector,
cf.~\eqref{eq:sym06}. Then one of two
following cases holds:
\begin{itemiz}
\item 	either $\pi^{(p)}(h)=1$ for all $h\in G_a$, and then  
$\Tr \Pip = 1$;
\item 	or $\sum_{h\in G_a} \pi^{(p)}(h) = 0$, and then 
$\Tr \Pip = 0$. 
\end{itemiz}
\end{lemma}

Let us call \defwd{active} (with respect to the representation $\pi^{(p)}$) 
the orbits $A_i$ such that $\Tr \Pip = 1$, and \defwd{inactive} the other
orbits. The restriction $L^{(p)}$ of $L$ to the subspace $P^{(p)}\C^n$ has 
dimension equal to the number of active orbits, and the following
result describes its matrix elements.

\begin{prop}[Matrix elements %of $L^{(p)}$ 
for an irreducible representation %$\pi^{(p)}$ 
of dimension $1$] 
For each orbit $A_i$ fix an $a_i\in A_i$. 
The subspace $P^{(p)}\C^n$ is spanned by the vectors 
$(\uip{i})_{A_i\;\rm{ active}}$ with components 
\begin{equation}
 \label{sym:10} 
 (\uip{i})_a = 
 \begin{cases}
  \cc{\pi^{(p)}(h)} & \text{if $a=h(a_i)\in A_i$\;,} \\
  0 & \text{otherwise\;.}
 \end{cases}
\end{equation} 
The off-diagonal matrix elements of $L^{(p)}$ between two active orbits $A_i$
and $A_j$ are again given by 
\begin{equation}
 \label{eq:sym11}
 L^{(p)}_{ij} = 
 \frac{\pscal{\uip{i}}{L\uip{j}}}{\pscal{\uip{i}}{\uip{i}}} = L^{(0)}_{ij} = 
 \frac{c^*_{ij}}{m^*_i} \e^{-h^*(A_i,A_j)/\eps} [1+\Order{\e^{-\theta/\eps}}]
 \;. 
\end{equation}
The diagonal elements of $L^{(p)}$ are given by 
\begin{equation}
 \label{eq:sym13}
 L^{(p)}_{ii} =  L^{(0)}_{ii} - \sum_{gG_{a_i}\in G/G_{a_i} \setminus G_{a_i}}
(1-\pi^{(p)}(g))L_{a_ig(a_i)}\;.
\end{equation} 
\end{prop}

Using Assumption~\ref{assump_sym_nondeg}, we can obtain a more explicit
expression for the diagonal matrix elements. For each orbit $A_i$, we can define
a unique \defwd{successor} $s(i)$, which labels the orbit which is easiest to
reach in one step from $A_i$: 
\begin{equation}
 \label{eq:sym10c}
 \inf_{j\neq i} h^*(A_i,A_j) = h^*(A_i,A_{s(i)})\;.
\end{equation} 
As a consequence of~\eqref{eq:sym10b}, we have 
\begin{equation}
 \label{eq:sym10d}
 L^{(0)}_{ii} = -\frac{c^*_{is(i)}}{m^*_i}
\e^{-h^*(A_i,A_{s(i)})/\eps}[1 + \Order{\e^{-\theta/\eps}}]\;.
\end{equation}
There are different cases to be considered, depending on whether it is easier,
starting from $a\in A_i$, to reach states outside $A_i$ or in $A_i\setminus a$.
Let $a^*$ be such that $h(a,a^*)=\inf_b h(a,b)$. Then 
\begin{equation}
 \label{eq:sym14}
 L^{(p)}_{ii} = 
 \begin{cases}
 L^{(0)}_{ii} [1 + \Order{\e^{-\theta/\eps}}]
 & \text{if $a^* \notin A_i$\;,} \\
 - 2 \brak{1-\re \pi^{(p)}(k)} L_{aa^*} [1 +
\Order{\e^{-\theta/\eps}}] 
 & \text{if $a^*=k(a)\in A_i$ and $k\neq k^{-1}$\;,} \\
 - \brak{1-\pi^{(p)}(k)} L_{aa^*} [1 + \Order{\e^{-\theta/\eps}}] 
 & \text{if $a^*=k(a)\in A_i$ and $k= k^{-1}$\;.}
 \end{cases}
\end{equation} 

Relation~\eqref{eq:sym13} and the fact that not all orbits are active for a
nontrivial representation imply that the matrix $L^{(p)}$ is not a generator if
$p\neq0$. We can however add a cemetery state to the set of active orbits, and
thus associate to $L^{(p)}$ a Markovian jump process on the augmented space. The
cemetery state is absorbing, which reflects the fact that all nonzero initial
conditions in $P^{(p)}\C^n$ are asymmetric and will converge to the symmetric
invariant distribution. 

\begin{theorem}[Eigenvalues associated with nontrivial irreducible
representations of dimension $1$]
\label{thm_dim1_nontrivial} 
Let $\pi^{(p)}$ be a nontrivial irreducible representation of $G$ of dimension
$1$, and let $n_p$ be the number of active orbits associated with $\pi^{(p)}$.
For sufficiently small $\eps$, the spectrum of $L^{(p)}$ consists in 
$n_p$ eigenvalues of geometric multiplicity $1$. They can be determined by
applying Theorem~\ref{thm_asym} to the augmented process defined by $L^{(p)}$,
and ignoring the eigenvalue~$0$. 
\end{theorem}

\subsection{Irreducible representations of dimension larger than $1$}

We finally turn to the computation of eigenvalues associated with irreducible
representations of higher dimension, which is more involved. The following
lemma is an analogue of Lemma~\ref{lem_dim1}, specifying which orbits will
appear in the image of the projector associated with a given representation. 

\begin{lemma}
\label{lem_dimd} 
Let $\pi^{(p)}$ be an irreducible representation of $G$ of dimension $d\geqs2$, 
and let $A_i$ be an orbit of $G$. Denote by $\pi_i(g)$ the permutation induced
by $g\in G$ on $A_i$, and let $\Pip$ be the associated projector,
cf.~\eqref{eq:sym06}. Then for arbitrary $a\in A_i$,
\begin{equation}
 \label{eq:sym21}
 \Tr(\Pip) = d\alpha_i^{(p)}\;, \qquad 
 \alpha_i^{(p)} = 
 \frac{1}{\abs{G_a}} \sum_{h\in G_a} \chi^{(p)}(h) 
 \in \set{0,1,\dots,d}\;.
\end{equation} 
Here $\chi^{(p)}(h) = \Tr \pi^{(p)}(h)$ denotes the characters of the
irreducible representation. 
\end{lemma}

Let us again call \defwd{active} (with respect to the irreducible
representation $\pi^{(p)}$) those orbits for which $\Tr(\Pip) > 0$. 

\begin{prop}[Matrix elements of $L^{(p)}$ for an irreducible representation
$\pi^{(p)}$ of dimension larger than $1$]
\label{prop_matrix_Lp} 
%Fix $a_i\in A_i$ for each orbit $A_i$. 
The subspace $P^{(p)}\C^n$ is spanned by the vectors 
$(u^a_i)_{i=1,\dots,m,a\in A_i}$ with components 
\begin{equation}
 \label{sym:14} 
 (u_i^a)_b = 
 \begin{cases}
  \displaystyle
  \frac{d}{\abs{G_a}} \sum_{g\in G_a}
  \cc{\chi^{(p)}(gh)} & \text{if $b=h(a)\in A_i$\;,} \\
  0 & \text{otherwise\;.}
 \end{cases}
\end{equation}
The matrix elements of $L^{(p)}$ between two different active orbits $A_i$
and $A_j$ are given by 
\begin{equation}
 \label{eq:sym15}
 \frac{\pscal{u_i^{h_1(a)}}{Lu_j^{h_2(b)}}}{\pscal{u_i^{h_1(a)}}{u_i^{h_1(a)}}}
= 
 \frac{c^*_{ij}}{\alpha_i^{(p)} m^*_i} \e^{-h^*(A_i,A_j)/\eps}
M^{(p)}_{h_1(a)h_2(b)}[1+\Order{\e^{-\theta/\eps}}]\;, 
\end{equation}
where $a\in A_i$, $b\in A_j$, $h_1, h_2\in G$, 
and 
\begin{equation}
 \label{eq:sym16}
 M^{(p)}_{h_1(a)h_2(b)} = \frac{1}{\abs{G_aG_b}} \sum_{g\in G_aG_b}
 \chi^{(p)}(h_1gh_2^{-1})\;.
\end{equation} 
The diagonal blocks of $L^{(p)}$ are given by the following expressions. 
Let $a\in A_i$ and let $a^*$ be such that $h(a,a^*)=\inf_b h(a,b)$. Then 
\begin{equation}
 \label{eq:sym17}
 \frac{\pscal{u_i^{h_1(a)}}{Lu_i^{h_2(a)}}}{\pscal{u_i^{h_1(a)}}{u_i^{h_1(a)}}}
= 
\begin{cases}
\displaystyle
 \frac{L^{(0)}_{ii}}{\alpha_i^{(p)}}
M^{(p)}_{h_1(a)h_2(a)}[1+\Order{\e^{-\theta/\eps}}]
&\text{if $a^*\notin A$}\;, \\
\\
\displaystyle
-\frac{L_{aa^*}}{\alpha_i^{(p)}}
M^{(p)}_{h_1(a)h_2(a)}[1+\Order{\e^{-\theta/\eps}}]
&\text{if $a^*\in A$}\;, 
\end{cases}
\end{equation}
where 
\begin{equation}
 \label{eq:sym18a}
 M^{(p)}_{h_1(a)h_2(a)} = \frac{1}{\abs{G_a}} \sum_{g\in G_a}
\chi^{(p)}(h_1gh_2^{-1})  
\end{equation} 
if $a^*\notin A_i$, while for $a^*=k(a)\in A_i$,
\begin{equation}
 \label{eq:sym18b}
 M^{(p)}_{h_1(a)h_2(a)} = 
 \begin{cases}
%  \displaystyle
%  \frac{1}{\abs{G_a}} \sum_{g\in G_a} \chi^{(p)}(h_1gh_2^{-1}) 
%  & \text{if $a^*\notin A_i$\;,} \\
 \displaystyle
 \frac{1}{\abs{G_a}} \sum_{g\in G_a} \bigbrak{2\chi^{(p)}(h_1gh_2^{-1}) 
 - \chi^{(p)}(h_1kgh_2^{-1}) - \chi^{(p)}(h_1k^{-1}gh_2^{-1})} 
 & \text{if $k\neq k^{-1}$,} \\
 \displaystyle
 \frac{1}{\abs{G_a}} \sum_{g\in G_a} \bigbrak{\chi^{(p)}(h_1gh_2^{-1}) -
\chi^{(p)}(h_1kgh_2^{-1})} 
 & \text{if $k= k^{-1}$.} 
 \end{cases}
\end{equation}
\end{prop}

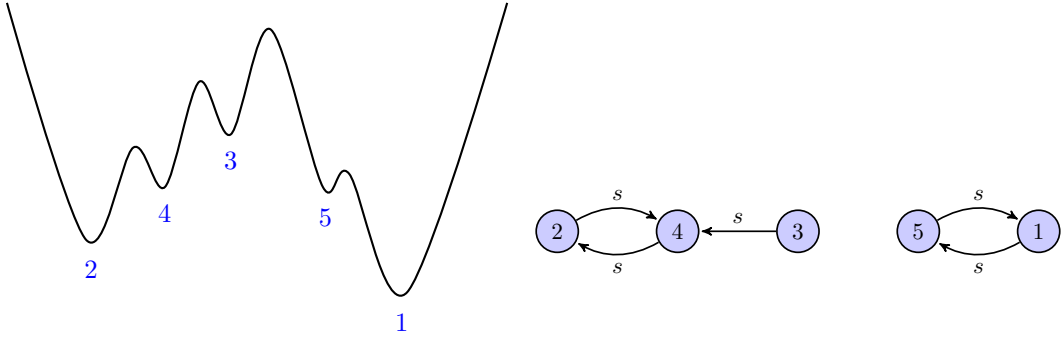
\begin{figure}
\begin{center}
\input{quintuple_well}
\scalebox{0.8}{{\raise 10mm \hbox{
\begin{tikzpicture}[->,>=stealth',shorten >=1pt,auto,node distance=2cm,
  thick,main
node/.style={circle,minimum size=0.7cm,fill=blue!20,draw}]
%\node at (-1.0,2.8) {Rule 1};
\node[main node] (2) {$2$};
\node[main node] (4) [right of=2] {$4$};
\node[main node] (3) [right of=4] {$3$};
\node[main node] (5) [right of=3] {$5$};
\node[main node] (1) [right of=5] {$1$};
\draw (2) edge[bend left] node[above] {$s$} (4);
\draw (4) edge[bend left] node[below] {$s$} (2);
\draw (3) edge node[above] {$s$} (4);
\draw (5) edge[bend left] node[above] {$s$} (1);
\draw (1) edge[bend left] node[below] {$s$} (5);
\end{tikzpicture}}}}
\end{center}
\vspace{-3mm}
\caption[]{Example of a graph of successors, with two cycles $(2,4)$ and
$(1,5)$, and associated potential.}
\label{fig:successors} 
\end{figure}

In order to apply this result, we have to choose, for each orbit $A_i$,
$d\alpha_i^{(p)}$ linearly independent vectors among the $(u^a_i)_{a\in
A_i}$. 

This result shows that in an appropriate basis, the matrix $L^{(p)}$ has a block
structure, with one block $L^{(p)}_{ij}$ for each pair of active orbits. Each
block has the same exponential weight as in the one-dimensional case, but the
prefactors are multiplied by a nontrivial matrix $M^{(p)}$ depending only on the
representation and on the stabilisers of the two orbits. 

In order to determine the eigenvalues, recall the
definition~\eqref{eq:sym10c} of the successor $A_{s(i)}$ of an orbit $A_i$. We
define an oriented graph on the set of orbits, with oriented edges $i\to s(i)$
(see~\figref{fig:successors}). Assumption~\ref{assump_sym_nondeg} implies that
each orbit is either in a cycle of length $2$, or in no cycle. If $i$ is in a
cycle of length $2$ and $V_i < V_{s(i)}$, we say that $i$ is \defwd{at the
bottom of a cycle}. We will make the following on-degeneracy assumption:

\begin{assump}
\label{assump_dimd_successor} 
Whenever $(i,j)$ form a cycle in the graph of successors, $L^{(p)}_{jj}$ is
invertible and the leading coefficient of the matrix 
\begin{equation}
 \label{eq:sym19}
 L^{(p)}_{ii} - L^{(p)}_{ij} \bigpar{L^{(p)}_{jj}}^{-1}L^{(p)}_{ji}
\end{equation} 
has the same exponent as the leading coefficient of $L^{(p)}_{ii}$. 
\end{assump}

Note that this assumption will not hold in the one-dimensional case
whenever $j=s(i)$. The reason it holds generically in the present case is that
there is no particular reason that the matrix $L^{(p)}$ is a generator. In fact
the row sums will typically be different from zero for each active orbit, which
can be viewed as the fact that each active orbit communicates with a cemetery
state. We will give an example in the next section. Under this assumption, we
obtain the following characterisation of eigenvalues. 

\begin{theorem}[Eigenvalues associated with representations of dimension larger
than $1$]
\label{thm_dimd} 
If Assumptions~\ref{assump_sym_meta}, \ref{assump_sym_nondeg}
and~\ref{assump_dimd_successor} hold and $\eps$ is small enough, then
the spectrum of $L^{(p)}$ consists, up to multiplicative errors
$1+\Order{\e^{-\theta/\eps}}$, in 
\begin{itemiz}
\item 	the eigenvalues of the matrices~\eqref{eq:sym19} for all orbits $A_i$
such that $i$ is at the bottom of a cycle and $s(i)$ is active;
\item 	the eigenvalues of all other diagonal blocks $L^{(p)}_{ii}$.
\end{itemiz}
\end{theorem}

%%%%%%%%%%%%%%%%%%%%%%%%%%%%%%%%%%%%%%%%%%%%%%%%%%%%%%%%%%%%%%%%%%%%%%%%%%%%%%

\subsection{Clustering of eigenvalues and computational cost}
\label{ssec:clustering} 

The results in the previous subsections show that the eigenvalues of the
generator $L$ are of the form $\lambda_k=-C_k \e^{-H_k/\eps}$, where the
Arrhenius exponents $H_k$ can only be of two (not necessarily disjoint) types:
\begin{enum}
\item 	either $H_k$ is a communication height occurring in the metastable
hierarchy of the orbits $A_1,\dots,A_{n_G}$;
\item 	or $H_k$ is of the form $h^*(A_i,A_{s(i)})$, where $s(i)$ is the
successor of an orbit; this case occurs for irreducible representations of
dimension $d\geqs 2$, or for one-dimensional irreducible representations if the
orbit $A_{s(i)}$ is inactive.
\end{enum}
As a consequence, we see a phenomenon of \emph{clustering of eigenvalues}:
there will in general be groups of several eigenvalues sharing the same
Arrhenius exponent $H_k$, but with possibly different prefactors $C_k$.

Note that Arrhenius exponents of the second type may be different from those
arising from the trivial representation, which are the only ones seen for
$G$-invariant initial distributions. The associated eigenvalues govern the
relaxation of initial distributions that are not uniform on each orbit towards
the stationary distribution, which is uniform on each orbit.

The computational cost to determine all $n$ eigenvalues can be estimated as
follows:
\begin{itemiz}
\item 	The cost for determining the $n_G$ eigenvalues associated with the
trivial representation is $\Order{n_G^2}$ at most, as a consequence of the
estimates obtained in Section~\ref{ssec_asymcost}.  
\item 	To determine all eigenvalues associated with nontrivial one-dimensional
irreducible representations, one can reuse the graph of successors already
computed in the previous step. The only new thing to be done is to determine
the active orbits. The cost for this has order $r_1(G)n_G$, where $r_1(G)$ is
the number of irreducible representations of dimension~$1$.
\item 	For each irreducible representation $\pi^{(p)}$ of dimension $d\geqs 2$,
one can again reuse the already determined graph of successors. It remains to
compute the matrices $M^{(p)}$, at cost $\Order{n_G(\alpha^{(p)}d)^2}$, and to
diagonalise the blocks $L^{(p)}_{ii}$ or \eqref{eq:sym19}. The exact cost is
hard to determine, but it can in any case be written as $\Order{\beta(G)n_G}$,
where $\beta(G)$ depends only on the symmetry group $G$. 
\end{itemiz}
As a result, we obtain the estimate 
\begin{equation}
 \label{eq:clustering02}
 \bigOrder{n_G[n_G + r_1(G) + \beta(G)]}
\end{equation} 
for the total cost of determining all $n$ eigenvalues, where $r_1(G)$ and
$\beta(G)$ depend only on the group $G$. Note that $r_1(G)$ cannot exceed the
order $\abs{G}$ of the group, because of the completeness relation $\sum_d
r_d(G)d^2=\abs{G}$.

If we consider a sequence of processes with a fixed symmetry group $G$ and
increasing number of states $n$, the cost will be $\Order{n_G}$. We expect that
typically, $n_G=\Order{n/\abs{G}}$, though one can construct counterexamples
where this number is larger. One can also encounter situations where $\abs{G}$
grows with $n$; we will discuss such a case in Section~\ref{ssec:costN}.

%%%%%%%%%%%%%%%%%%%%%%%%%%%%%%%%%%%%%%%%%%%%%%%%%%%%%%%%%%%%%%%%%%%%%%%%%%%%%%

%\newpage

\section{Examples}
\label{sec_ex}

We discuss in this section two applications of the previous results, which are
motivated by Example~\ref{ex:Kawasaki}. In Appendix~\ref{appendix}, we show how
the local minima and saddles of that system can be computed for small coupling
$\gamma$. Here we determine the eigenvalues of the associated markovian jump
processes, for the cases $N=4$, which is relatively simple and can be solved in
detail, and $N=8$, which is already substantially more involved (the case $N=6$
features degenerate saddles, so we do not discuss it here). 

%%%%%%%%%%%%%%%%%%%%%%%%%%%%%%%%%%%%%%%%%%%%%%%%%%%%%%%%%%%%%%%%%%%%%%%%%%%%%%

\subsection{The case $N=4$}
\label{ssec:ex_N4} 

As explained in Appendix~\ref{appendix}, for $0\leqs\gamma<2/5$ the system
described in Example~\ref{ex:Kawasaki} admits $6$ local minima, connected by
$12$ saddles of index $1$. The potential~\eqref{meta04}  is invariant under the
symmetry group $G=D_4\times\Z_2=\langle r, s, c\rangle$, which has
order $16$ and is generated by 
the rotation $r:(x_1,x_2,x_3,x_4)\mapsto(x_2,x_3,x_4,x_1)$, the reflection 
$s:(x_1,x_2,x_3,x_4)\mapsto(x_4,x_3,x_2,x_1)$, and the sign change
$c:x\mapsto-x$. The local minima form two orbits 
\begin{align}
\nonumber
A_1 &= \bigset{(1,1,-1,-1), (1,-1,-1,1), (-1,-1,1,1),  (-1,1,1,-1)}
+ \Order{\gamma} \\
\nonumber
&= \bigset{a,r(a),r^2(a),r^3(a)} \\
A_2 &= \bigset{(1,-1,1,-1), (-1,1,-1,1)} + \Order{\gamma}
= \bigset{b,r(b)}\;,
\label{eq:ex4_01} 
\end{align}
where we have chosen $a=(1,1,-1,-1) + \Order{\gamma}$ and $b=(1,-1,1,-1) +
\Order{\gamma}$ as representatives. 
The associated stabilisers are 
\begin{align}
\nonumber
G_a &= \bigset{\id, r^2s, sc, r^2c} \\
G_b &= \bigset{\id, rs, r^2, r^3s, sc, rc, r^2sc, r^3c}\;.
\label{eq:ex4_02}
\end{align} 
The graph of connections forms an octaeder as shown in~\figref{fig:octaeder}. 
Note in particular that~\eqref{eq:sym04c} is satisfied. Indeed, $\abs{G_a\cap
G_b}=2$, and each site in $A_1$ has $\abs{G_a}/\abs{G_a\cap G_b}=2$ neighbours
in $A_2$, while each site in $A_2$ has $\abs{G_b}/\abs{G_a\cap G_b}=4$
neighbours in $A_1$. The associated graph in terms of orbits is also shown
in~\figref{fig:octaeder}. 

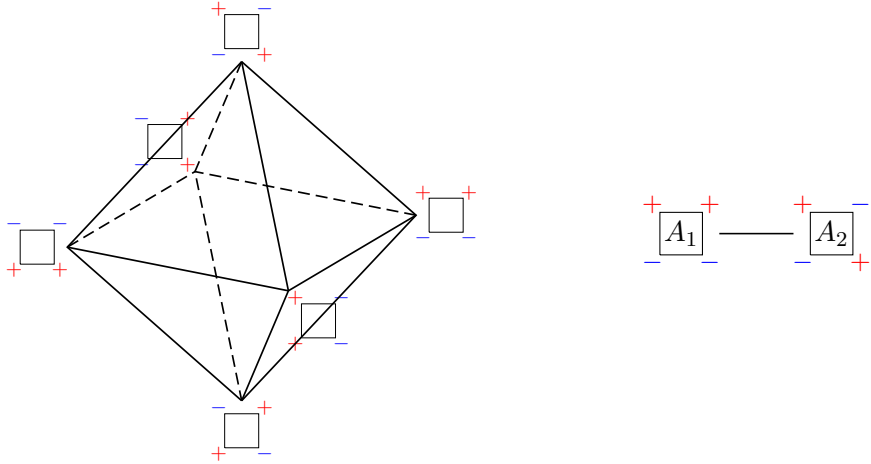
\begin{figure}
\begin{center}
\scalebox{0.8}{
\input{graphe_4}
}
\hspace{15mm}
{\raise 25mm \hbox{\input{graphe_4_orbit}}}
\end{center}
\vspace{-5mm}
\caption[]{The graph $\cG=(\cX,E)$ for the case $N=4$ has $6$ nodes and $12$
edges, forming an octaeder. The associated graph on the set of orbits has two
sites and one edge.}
\label{fig:octaeder} 
\end{figure}

The analysis of the potential~\eqref{meta04} shows that the transition
probabilities are of the form 
\begin{equation}
 \label{eq:ex4_03}
 L_{ab} = \frac{c_{ab}}{m_a} \e^{-h_{ab}/\eps}\;, \qquad
 L_{ba} = \frac{c_{ab}}{m_b} \e^{-h_{ba}/\eps}\;, \qquad
 L_{aa'} = \frac{c_{aa'}}{m_a} \e^{-h_{aa'}/\eps}\;, 
\end{equation} 
where the exponents satisfy (cf.\ Table~\ref{table:orbits_4} in
Appendix~\ref{appendix})  
\begin{equation}
 \label{eq:ex4_04}
 h_{ba} < h_{aa'} < h_{ab}
\end{equation} 
whenever $0 < \gamma < 2/5$. We set $\theta = (h_{aa'} - h_{ba})\wedge(h_{ab}
- h_{aa'})$.
The generator $L$ is of the form 
\begin{equation}
 \label{eq:ex4_05}
 L = 
 \begin{pmatrix}
 L^{11} & L^{12} \\ L^{21} & L^{22}
 \end{pmatrix}\;,
\end{equation} 
with blocks 
\begin{gather}
 \label{eq:ex4_06}
 L^{11} = 
 \begin{pmatrix}
 -2L_{aa'}-2L_{ab} & L_{aa'} & 0 & L_{aa'} \\
 L_{aa'} & -2L_{aa'}-2L_{ab} & L_{aa'} & 0 \\
 0 & L_{aa'} & -2L_{aa'}-2L_{ab} & L_{aa'} \\
 L_{aa'} & 0 & L_{aa'} & -2L_{aa'}-2L_{ab} \\
 \end{pmatrix}\;, \\
 L^{12} = 
 \begin{pmatrix}
 L_{ab} & L_{ab} \\
 L_{ab} & L_{ab} \\
 L_{ab} & L_{ab} \\
 L_{ab} & L_{ab} 
 \end{pmatrix}\;, \qquad
 L^{21} = 
 \begin{pmatrix}
 L_{ba} & L_{ba} & L_{ba} & L_{ba} \\
 L_{ba} & L_{ba} & L_{ba} & L_{ba} 
 \end{pmatrix}
\;, \qquad
 L^{22} = 
 \begin{pmatrix}
 -4L_{ba} & 0 \\
 0 & -4L_{ba}
 \end{pmatrix}\;.
 \nonumber
\end{gather} 
We can now apply the results of Section~\ref{sec_res}. From the known
irreducible representations of the dihedral group $D_4$ (cf.\
Example~\ref{ex:dihedral}) and the fact that $c$ commutes with $r$ and $s$, we
deduce that $G$ has $8$ irreducible representations of dimension $1$, given by 
\begin{equation}
 \label{eq:ex4_07}
 \pi_{\rho\sigma\tau}(r^is^jc^k) = \rho^i\sigma^j\tau^k\;, 
 \qquad
 \rho, \sigma, \tau = \pm1\;,
\end{equation} 
and two irreducible representation of dimension $2$, that we denote
$\pi_{1,\pm}$, with characters 
\begin{equation}
 \label{eq:ex4_08}
 \chi_{1,\pm}(r^is^jc^k) = 2\cos(i\pi/2) \delta_{j0} (\pm1)^k\;.
\end{equation} 
Applying Lemma~\ref{lem_dim1} and Lemma~\ref{lem_dimd}, we obtain  
Table~\ref{table:active_4} for active and inactive orbits. 

\begin{table}[ht]
\begin{center}
\begin{tabular}{|c||c|c||r|}
\hline
\hlinespace
 & $A_1$ & $A_2$ & $\alpha d$ \\
\hline
\hline
\hlinespace
$\pi_{+++}$ & $1$ & $1$ & $2$ \\
$\pi_{++-}$ & $0$ & $0$ & $0$ \\
$\pi_{+-+}$ & $0$ & $0$ & $0$ \\
$\pi_{+--}$ & $0$ & $0$ & $0$ \\
$\pi_{-++}$ & $1$ & $0$ & $1$ \\
$\pi_{-+-}$ & $0$ & $0$ & $0$ \\
$\pi_{--+}$ & $0$ & $0$ & $0$ \\
$\pi_{---}$ & $0$ & $1$ & $1$ \\
\hline
\hlinespace
$\pi_{1,+}$ & $0$ & $0$ & $0$ \\
$\pi_{1,-}$ & $2$ & $0$ & $2$ \\
\hline
\hline
\hlinespace
$\abs{A}$ & $4$ & $2$ & $6$ \\
\hline
\end{tabular}
\end{center}
\caption[]{Active and inactive orbits and number of eigenvalues for the
different irreducible representations when $N=4$.}
\label{table:active_4} 
\end{table}

Table~\ref{table:active_4} shows that the permutation representation $\pi$
induced by $G$ on $\cX = A_1\cup A_2$ admits the decomposition
\begin{equation}
 \label{eq:ex4_09}
 \pi = 2\pi_{+++} \oplus \pi_{-++} \oplus \pi_{---} \oplus \pi_{1,-}\;.
\end{equation} 
We can now determine the eigenvalues associated with each irreducible
representation:

\begin{itemiz}
\item 	Trivial representation $\pi_{+++}$: 
The associated subspace has dimension $2$, and is spanned by the vectors 
$\transpose{(1,1,1,1,0,0)}$ and $\transpose{(0,0,0,0,1,1)}$. The matrix in this
basis is given by 
\begin{equation}
 \label{eq:ex4_10}
 L^{(0)} = 
 \begin{pmatrix}
 -2L_{ab} & 2L_{ab} \\ 4L_{ba} & -4L_{ba}
 \end{pmatrix}\;,
\end{equation} 
as can be checked by a direct computation, and is compatible with
Proposition~\ref{prop:L_trivial}. The eigenvalues of $L^{(0)}$ are $0$ and 
$-4L_{ba}-2L_{ab}$, which is also compatible with
Theorem~\ref{thm_dim1_trivial} (giving the leading-order behaviour
$-4L_{ba}[1+\Order{\e^{-\theta/\eps}}]$). In particular, we conclude that if
$\mu$ is the uniform distribution on $A_2$, then we have the Eyring--Kramers
formula 
\begin{equation}
 \label{eq:ex4_11}
 \bigexpecin{\mu}{\tau_{A_1}} = 
 \frac14 \frac{m_b}{c_{ab}} \e^{h_{ba}/\eps} 
 \bigbrak{1+\Order{\e^{-\theta/\eps}}}\;.
\end{equation} 
Note the prefactor $1/4$, which is due to the symmetry. 

\item 	Representation $\pi_{-++}$: 
From~\eqref{sym:10} we see that the associated subspace is spanned by the
vector $\transpose{(1,-1,1,-1,0,0)}$. A direct computation shows that the
corresponding eigenvalue is $-4L_{aa'}-2L_{ab}$, which is also compatible with
\eqref{eq:sym14}, where we have to apply the second case, and use the fact that 
$\pi_{-++}(r)=-1$. 

\item 	Representation $\pi_{---}$: 
From~\eqref{sym:10} we see that the associated subspace is spanned by the
vector $\transpose{(0,0,0,0,1,-1)}$. A direct computation shows that the
corresponding eigenvalue is $-4L_{ba}$, and the same result is obtained by
applying~\eqref{eq:sym14} (first case). 

\item 	Representation $\pi_{1,-}$: 
From~\eqref{sym:14} we obtain that the associated subspace is spanned by 
the vectors $\transpose{(2,0,-2,0,0,0)}$ and $\transpose{(0,2,0,-2,0,0)}$. The
associated matrix is 
\begin{equation}
 \label{eq:ex4_12}
 L^{1,-} = 
 \begin{pmatrix}
 -2L_{aa'}-2L_{ab} & 0 \\ 0 & -2L_{aa'}-2L_{ab} 
 \end{pmatrix}
\end{equation}
and thus $-2L_{aa'}-2L_{ab}$ is an eigenvalue of multiplicity $2$. 
The leading-order behaviour $-2L_{aa'}[1+\Order{\e^{-\theta/\eps}}]$ is also
obtained using~\eqref{eq:sym17} with $a^*=r(a)$ and~\eqref{eq:sym18b} (first
case), which shows that $M=2\one$. 
\end{itemiz}

In summary, to leading order the eigenvalues of the generator are given by 
\begin{equation}
 \label{eq:ex4_13}
 0\;,\; -2L_{aa'}\;,\; -2L_{aa'}\;,\; -4L_{aa'}\;,\; -4L_{ba}\;,\; -4L_{ba}\;.
\end{equation} 
They form three clusters sharing an exponent, with possibly different
prefactors. Note in particular that the spectral gap is given by $2L_{aa'}$,
which is smaller than in the case of an asymmetric double-well, where it would
be $L_{ba}$. This is due to the fact that the slowest process in the system is
the internal dynamics of the orbit $A_1$. 

%%%%%%%%%%%%%%%%%%%%%%%%%%%%%%%%%%%%%%%%%%%%%%%%%%%%%%%%%%%%%%%%%%%%%%%%%%%%%%

\subsection{The case $N=8$}
\label{ssec:ex_N8} 

In the case $N=8$, the potential is invariant under the group
$G=D_8\times\Z_2$, which has order $32$. As explained in
Appendix~\ref{appendix}, there are $182$ local minima, connected by $560$
saddles of index $1$. The local minima form $12$ orbits, of cardinality varying
between $2$ and~$32$ depending on the size of their stabiliser, see
Table~\ref{table:orbits_8}. 

\begin{table}[ht]
\begin{center}
\begin{tabular}{|c|r|c|l|}
\hline
\hlinespace
$A$ & $\abs{A}$ & $a$ & $G_a$ \\
\hline 
\hline 
\hlinespace
$A_1$ &  $8$ & $(+,+,+,+,-,-,-,-)$ & $\set{\id,r^4s,r^4c,sc}$ \\
$A_2$ &  $4$ & $(+,+,-,-,+,+,-,-)$ &
$\set{\id,r^2s,r^4,r^6s,sc,r^2c,r^4sc,r^6c}$ \\
$A_3$ & $16$ & $(+,+,+,-,-,+,-,-)$ & $\set{\id,r^3s}$ \\
$A_4$ & $16$ & $(+,-,-,-,+,+,+,-)$ & $\set{\id,sc}$ \\
$A_5$ &  $8$ & $(+,-,-,+,-,+,+,-)$ & $\set{\id,r^4s,r^4c,sc}$ \\
$A_6$ & $16$ & $(+,+,-,+,-,+,-,-)$ & $\set{\id,sc}$ \\
$A_7$ &  $2$ & $(+,-,+,-,+,-,+,-)$  &
$\{\id,rs,r^2,r^3s,r^4,r^5s,r^6,r^7s,$ \\
 &  &  &
$\phantom{\{}sc,rc,r^2sc,r^3c,r^4sc,r^5c,r^6sc,r^7c\}$ \\
\hline 
\hlinespace
$A_8$ & $32$ & $(\alpha,\alpha,\beta,\beta,\beta,\alpha,\beta,\beta)$ & 
$\set{\id}$ \\
$A_9$ & $16$ & $(\alpha,\alpha,\alpha,\beta,\beta,\beta,\beta,\beta)$ & 
$\set{\id,r^3s}$ \\
$A_{10}$ & $16$ & $(\beta,\alpha,\beta,\beta,\alpha,\beta,\alpha,\beta)$ & 
$\set{\id,r^3s}$ \\
$A_{11}$ & $16$ & $(\beta,\alpha,\beta,\alpha,\beta,\beta,\beta,\alpha)$ & 
$\set{\id,r^3s}$ \\
$A_{12}$ & $32$ & $(\alpha,\beta,\alpha,\alpha,\beta,\beta,\beta,\beta)$ & 
$\set{\id}$ \\
\hline
\end{tabular}
\end{center}
\caption[]{Orbits $A_i$ for the case $N=8$ with their
cardinality, one representative $a$ and its stabiliser $G_a$. The symbols $\pm$
stand for $\pm1+\Order{\gamma}$, while $\alpha=\pm5/\sqrt{19}+\Order{\gamma}$
and $\beta=\mp3/\sqrt{19}+\Order{\gamma}$. Stabilisers of other elements
$a'=g(a)$ of any orbit are obtained by conjugation with $g$.}
\label{table:orbits_8} 
\end{table}

Local minima occur in two types:
\begin{itemiz}
\item 	those with $4$ coordinates equal to $1+\Order{\gamma}$, and $4$
coordinates equal to $-1+\Order{\gamma}$; we denote these coordinates $+$ and
$-$;
\item 	and those with $3$ coordinates equal to $\pm\alpha$ and $5$ coordinates
equal to $\pm\beta$, where $\alpha=5/\sqrt{19}+\Order{\gamma}$ and
$\beta=-3/\sqrt{19}+\Order{\gamma}$. 
\end{itemiz}

\begin{figure}[t]
%\vspace{-15mm}
\begin{center}
\scalebox{0.9}{
\hspace{-22mm}
\input{graphe_8_labels_3}
}
\end{center}
\vspace{-18mm}
\caption[]{Graph on the set of orbits for the case $N=8$. Each node displays a
particular representative of the orbit. Figures next to the edges denote the
total number of connections between elements of the orbits.
Note that there is a kind of hysteresis effect, in the sense that going around a
loop on the graph, following the connection rules, one does not necessarily end
up with the same representative of the orbit.}
\label{fig:orbits_N8} 
\end{figure}
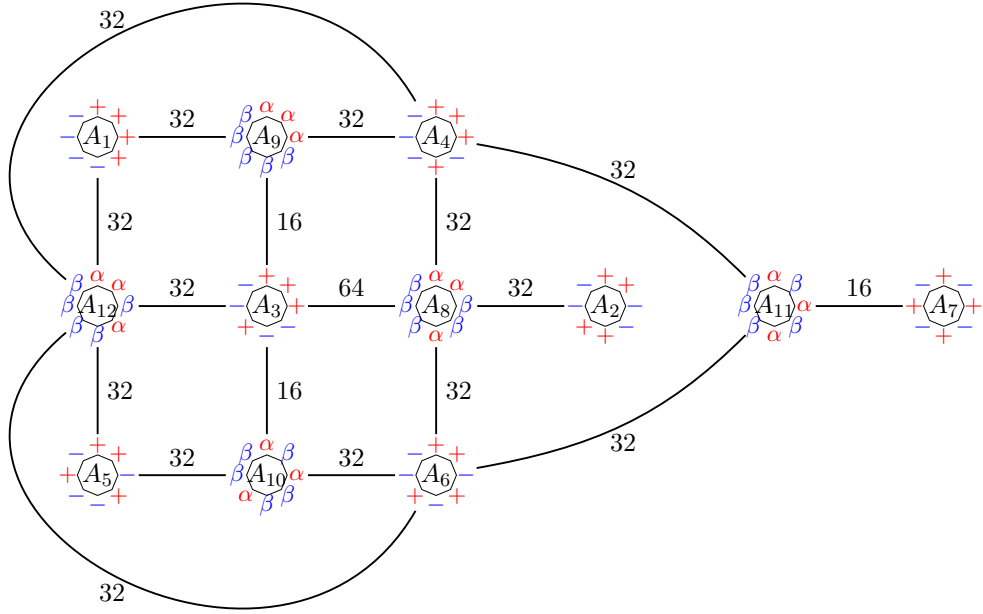

These local minima are connected according to the following rules: 
\begin{align}
\nonumber
&3\times(\alpha\longleftrightarrow+) & & 3\times(-\alpha\longleftrightarrow-) \\
\label{eq:connection_rules} 
&1\times(\beta\longleftrightarrow+)  & & 1\times(-\beta\longleftrightarrow-) \\
\nonumber
&4\times(\beta\longleftrightarrow-)  & & 4\times(-\beta\longleftrightarrow+)\;, 
\end{align}
meaning that each $\alpha$ and one of the $\beta$s are connected with a $+$, and
so on. A major simplification will arise from the fact that there are no
connections among sites of a same orbit. 

We do not attempt to draw the full graph, which has $182$ vertices and $560$
edges. However, \figref{fig:orbits_N8} shows the graph on the set of orbits. The
metastable hierarchy has been established by computing the height of saddles to
second order in $\gamma$ with the help of computer algebra.

The group $G=D_8\times\Z_2$ has again $8$ irreducible representations of
dimension $1$, given by 
\begin{equation}
 \label{eq:ex8_01}
 \pi_{\rho\sigma\tau}(r^is^jc^k) = \rho^i\sigma^j\tau^k\;, 
 \qquad
 \rho, \sigma, \tau = \pm1\;,
\end{equation} 
In addition, it has $6$ irreducible representations of dimension $2$ deduced
from those of $D_8$, cf.~\eqref{eq:Dn_2dim}. We denote them 
$\pi_{l,\pm}$, $l=1,2,3$, and their characters satisfy
(see~\eqref{eq:Dn_characters})  
\begin{equation}
 \label{eq:ex8_02}
 \chi_{l,\pm}(r^is^jc^k) = 2\cos(il\pi/4) \delta_{j0} (\pm1)^k\;.
\end{equation}
Applying Lemma~\ref{lem_dim1} and Lemma~\ref{lem_dimd}, we obtain
Table~\ref{table:active_8} of active and inactive orbits.

\begin{table}[th]
\begin{center}
\begin{tabular}{|c||c|c|c|c|c|c|c|c|c|c|c|c||r|}
\hline
\hlinespace
 & $A_1$ & $A_2$ & $A_3$ & $A_4$ & $A_5$ & $A_6$ & $A_7$ & $A_8$ & $A_9$ &
$A_{10}$ & $A_{11}$ & $A_{12}$ & $\alpha d$ \\
\hline
\hline
\hlinespace
$\pi_{+++}$ & $1$ & $1$ & $1$ & $1$ & $1$ & $1$ & $1$ & $1$ & $1$ & $1$ & $1$ &
$1$ & $12$\\
$\pi_{++-}$ & $0$ & $0$ & $1$ & $0$ & $0$ & $0$ & $0$ & $1$ & $1$ & $1$ & $1$ &
$1$ &  $6$\\
$\pi_{+-+}$ & $0$ & $0$ & $0$ & $0$ & $0$ & $0$ & $0$ & $1$ & $0$ & $0$ & $0$ &
$1$ &  $2$\\
$\pi_{+--}$ & $0$ & $0$ & $0$ & $1$ & $0$ & $1$ & $0$ & $1$ & $0$ & $0$ & $0$ &
$1$ &  $4$\\
$\pi_{-++}$ & $1$ & $1$ & $0$ & $1$ & $1$ & $1$ & $0$ & $1$ & $0$ & $0$ & $0$ &
$1$ &  $7$\\
$\pi_{-+-}$ & $0$ & $0$ & $0$ & $0$ & $0$ & $0$ & $0$ & $1$ & $0$ & $0$ & $0$ &
$1$ &  $2$\\
$\pi_{--+}$ & $0$ & $0$ & $1$ & $0$ & $0$ & $0$ & $0$ & $1$ & $1$ & $1$ & $1$ &
$1$ &  $6$\\
$\pi_{---}$ & $0$ & $0$ & $1$ & $1$ & $0$ & $1$ & $1$ & $1$ & $1$ & $1$ & $1$ &
$1$ &  $9$\\
\hline
\hlinespace
$\pi_{1,+}$ & $0$ & $0$ & $2$ & $2$ & $0$ & $2$ & $0$ & $4$ & $2$ & $2$ & $2$ &
4 & $20$\\
$\pi_{1,-}$ & $2$ & $0$ & $2$ & $2$ & $2$ & $2$ & $0$ & $4$ & $2$ & $2$ & $2$ &
4 & $24$\\
$\pi_{2,+}$ & $2$ & $0$ & $2$ & $2$ & $2$ & $2$ & $0$ & $4$ & $2$ & $2$ & $2$ &
4 & $24$\\
$\pi_{2,-}$ & $0$ & $2$ & $2$ & $2$ & $0$ & $2$ & $0$ & $4$ & $2$ & $2$ & $2$ &
4 & $22$\\
$\pi_{3,+}$ & $0$ & $0$ & $2$ & $2$ & $0$ & $2$ & $0$ & $4$ & $2$ & $2$ & $2$ &
4 & $20$\\
$\pi_{3,-}$ & $2$ & $0$ & $2$ & $2$ & $2$ & $2$ & $0$ & $4$ & $2$ & $2$ & $2$ &
4 & $24$\\
\hline
\hline
\hlinespace
$\abs{A}$ & $8$ & $4$ & $16$ & $16$ & $8$ & $16$ & $2$ & $32$ & $16$ & $16$ &
$16$ & $32$ & $182$ \\
\hline
\end{tabular}
\end{center}
\caption[]{Active and inactive orbits and number of eigenvalues for the
different irreducible representations when $N=8$. There are $182$ eigenvalues
in total, $48$ associated with $1$-dimensional irreducible representations, and
$134$ associated with $2$-dimensional irreducible representations.}
\label{table:active_8} 
\end{table}

It is now possible to determine the eigenvalues associated with each irreducible
representation. The trivial representation $\pi_{+++}$ will yield $12$
eigenvalues, which are given by Theorem~\ref{thm_dim1_trivial}. The only
difference with the Eyring--Kramers formula of the asymmetric case is an extra
factor of the form $\abs{G_c}/\abs{G_a\cap G_b}$, where $(a,b)$ is the highest
edge of an optimal path from $A_k$ to $\cM_{k-1}$, and $c\in A_k$. For
instance, the optimal path from $A_7$ to $\cM_{6}$ is $A_7\to A_{11}\to A_4$,
and its highest edge is $A_7\to A_{11}$. We thus obtain 
\begin{equation}
 \label{eq:ex8_03}
 \lambda^{(0)}_7 = 8 \frac{c_{a_7a_{11}}}{m_{a_7}} \e^{-H(A_7,A_{11})/\eps}
 \bigbrak{1+\Order{\e^{-\theta/\eps}}}
\end{equation}
where $a_7\in A_7$ and $a_{11}\in A_{11}$, because $\abs{G_{a_7}}=16$ and
$\abs{G_{a_7}\cap G_{a_{11}}}=2$ (cf.\ Table~\ref{table:orbits_8}). 

\begin{figure}[t]
\vspace{-15mm}
\begin{center}
\scalebox{0.9}{
\hspace{-22mm}
\input{graphe_8_ex1_3}
}
\end{center}
\vspace{-18mm}
\caption[]{Graph for the case $N=8$ associated with the representation
$\pi_{-++}$.}
\label{fig:orbits_N8_-++} 
\end{figure}
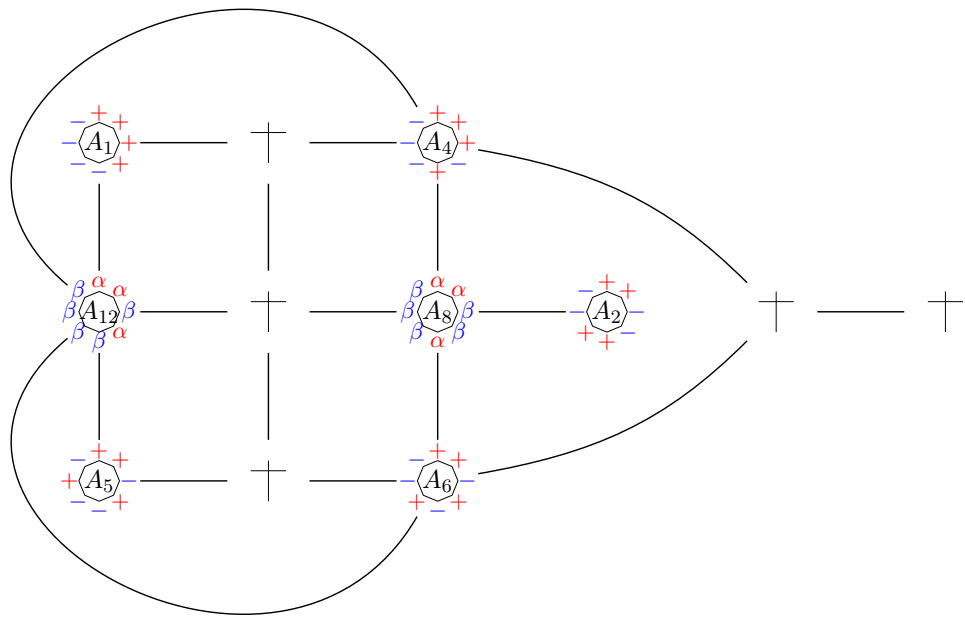

The eigenvalues associated with other irreducible representations of dimension
$1$ can be deduced from the metastable hierarchy of the corresponding set of
active orbits. For instance, \figref{fig:orbits_N8_-++} shows the graph
obtained for the representation $\pi_{-++}$, which yields $7$ eigenvalues. 
An important difference to the previous case arises from the fact that some
communication heights relevant for the eigenvalues are associated with
transitions to the cemetery state. In particular, $A_1$ is no longer at the
bottom of the hierarchy (which is occupied by the cemetery state), and thus
there will be an eigenvalue of order $\e^{-H(A_1,A_9)/\eps}$, because $A_9$ is
the successor of $A_1$, of the form 
\begin{equation}
 \label{eq:ex8_04}
 \lambda^{(-++)}_1 = -4 \frac{c_{a_1a_9}}{m_{a_1}} \e^{-H(A_1,A_9)/\eps}
 \bigbrak{1+\Order{\e^{-\theta/\eps}}}\;.
\end{equation}
Eigenvalues associated with irreducible representations of dimension $2$ are
given by Theorem~\ref{thm_dimd}. The graph of successors is shown in
\figref{fig:successors_8}. Observe that $A_1$ is at the bottom of the cycle
containing $A_1$ and $A_9$. For instance, for the representation $\pi_{1,-}$, 
applying Proposition~\ref{prop_matrix_Lp} with the choice of basis
$(u_i^a,u_i^{a'})$ with $a'=r^2(a)$ for each orbit, we obtain 
\begin{align}
\nonumber
 L_{11}^{(1,-)} &= L_{11}^{(0)} \one \brak{1+\Order{\e^{-\theta/\eps}}}\;, 
 & L_{11}^{(0)} = -4 \frac{c_{a_1a_9}}{m_{a_1}} \e^{-H(A_1,A_9)/\eps}\;,  \\
\nonumber
 L_{19}^{(1,-)} &= L_{11}^{(0)} M_{19} \brak{1+\Order{\e^{-\theta/\eps}}}\;, \\
\nonumber
 L_{91}^{(1,-)} &= L_{99}^{(0)} M_{91} \brak{1+\Order{\e^{-\theta/\eps}}}\;,  
 & L_{99}^{(0)} = -2 \frac{c_{a_9a_1}}{m_{a_9}} \e^{-H(A_9,A_1)/\eps}\;,  \\
 L_{99}^{(1,-)} &= L_{99}^{(0)} \one \brak{1+\Order{\e^{-\theta/\eps}}}\;,   
 \label{eq:ex8_05}
\end{align} 
where 
\begin{equation}
 \label{eq:ex8_06}
 M_{19} = \frac14
 \begin{pmatrix}
 2 + \sqrt{2} & \sqrt{2} \\ -\sqrt{2} & 2+\sqrt{2}
 \end{pmatrix}\;, 
 \qquad
  M_{91} = \frac14
 \begin{pmatrix}
 2 + \sqrt{2} & -\sqrt{2} \\ \sqrt{2} & 2+\sqrt{2}
 \end{pmatrix}\;. 
\end{equation} 
Thus by Theorem~\ref{thm_dimd}, the eigenvalues associated with $A_1$ are
those of the matrix 
\begin{equation}
 \label{eq:ex8_07}
 L^{(1,-)}_{11} - L^{(1,-)}_{19} \bigpar{L^{(1,-)}_{99}}^{-1}L^{(1,-)}_{91}
 = -(2-\sqrt{2}) \frac{c_{a_1a_9}}{m_{a_1}} \e^{-H(A_1,A_9)/\eps} 
 \one\brak{1+\Order{\e^{-\theta/\eps}}}\;.
\end{equation} 
We thus obtain a double eigenvalue given by the Eyring--Kramers law with an
extra factor of $(2-\sqrt{2})$. 

\begin{figure}[t]
%\vspace{-15mm}
\begin{center}
\scalebox{0.9}{
\hspace{-22mm}
\input{graphe_8_successeur}
}
\end{center}
\vspace{-18mm}
\caption[]{The graph of successors for $N=8$.}
\label{fig:successors_8} 
\end{figure}
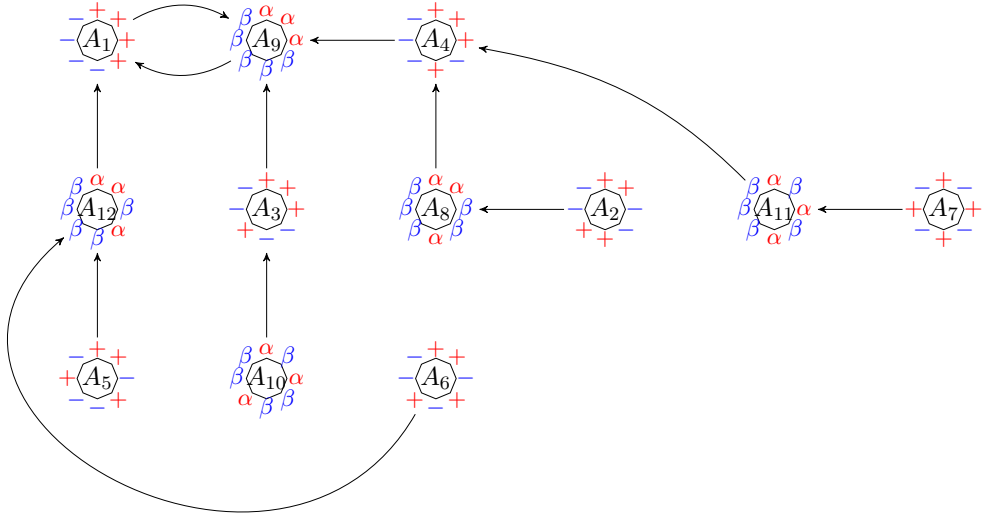

%%%%%%%%%%%%%%%%%%%%%%%%%%%%%%%%%%%%%%%%%%%%%%%%%%%%%%%%%%%%%%%%%%%%%%%%%%%%%%

\subsection{Computational cost for arbitrary even $N$}
\label{ssec:costN} 

For general even $N$, the number $n$ of local minima of the potential is of
order $2^N$ (see Appendix~\ref{appendix}), while the symmetry group $G$ has
order $4N$. However, all irreducible representations have dimension $1$ or $2$;
in fact (cf.\ Example~\ref{ex:dihedral}), there are exactly $r_1(G)=8$
irreducible representations of dimension $1$, and $r_2(G)=N-2$ irreducible
representations of dimension $2$. This implies that the coefficient $\beta(G)$
in~\eqref{eq:clustering02} satisfies $\beta(G)=\Order{N}$.

As a consequence, the computational cost for determining all eigenvalues is at
most $\Order{2^{2N}}$, 
where we only have used the trivial bound $n_G \leqs n=\Order{2^N}$. It is
possible that a more detailed analysis will reveal that actually $n_G =
\Order{2^N/N}$, which would decrease the cost by a factor $N^2$.

%%%%%%%%%%%%%%%%%%%%%%%%%%%%%%%%%%%%%%%%%%%%%%%%%%%%%%%%%%%%%%%%%%%%%%%%%%%%%%

%\phantom{M}

%\newpage

\section{Proofs -- Group theory}
\label{sec_proofG}

%%%%%%%%%%%%%%%%%%%%%%%%%%%%%%%%%%%%%%%%%%%%%%%%%%%%%%%%%%%%%%%%%%%%%%%%%%%%%%

In this section, we give the proofs of the different expressions for the matrix
elements of $L$ restricted to the subspaces $\Pip\C^n$ associated with the
irreducible representations $\pi^{(p)}$. Although we have introduced the
results by starting with the trivial representation, then moving to other
representations of dimension $1$, and finally to higher-dimensional
representations, it will be more straightforward to give directly proofs in the
case of a general irreducible representation of arbitrary dimension $d$, and
then to particularize to the cases $d=1$ and $p=0$. 

To simplify notations, we will fix
an irreducible representation $\pi=\pi^{(p)}$, orbits $A=A_i$, $B=A_j$, and
elements $a\in A$ and $b\in B$. We write $\alpha_i=\smash{\alpha_i^{(p)}}$ and 
$\chi=\chi^{(p)}$. Recall that $\pi_i(g)$ denotes the permutation matrix induced
by $g$ on the orbit $A$ (we will consider $\pi_i$ as a linear map on $\C^n$
which is identically zero on $\cX\setminus A$). The associated projector
$\PA=\Pip$ is given by
(cf.\ \eqref{eq:sym06})
\begin{equation}
 \label{eq:pgt00} 
 \PA = \frac{d}{\abs{G}} \sum_{g\in G} \cc{\chi(g)}\pi_i(g)\;.
\end{equation} 
The only nonzero matrix elements of $P_i$ are those between elements in $A$,
and they can be written as 
\begin{equation}
 \label{eq:pgt00a} 
 (\PA)_{a h(a)} = \frac{d}{\abs{G}} \sum_{g\in G_a} \cc{\chi(gh)} 
 \qquad \forall a\in A, \forall h\in G\;.
\end{equation} 
We write $\PB =\Pjp$ for the similarly defined projector associated with the
orbit $B$.

\begin{proof}[Proof of Lemma~\ref{lem_dimd}]
Taking the trace of~\eqref{eq:pgt00}, we obtain  
\begin{equation}
 \label{eq:pgt01}
 \alpha_i d = \Tr(\PA ) = \frac{d}{\abs{G}} \sum_{g\in G}\cc{\chi(g)}
\Tr(\pi_i(g))\;.
\end{equation} 
Note that $\Tr(\pi_i(g)) = \abs{A^g} = \abs{A}\indexfct{g\in G_a}$. Therefore, 
\begin{equation}
 \label{eq:pgt02}
 \alpha_i = \frac{\abs{A}}{\abs{G}} \sum_{g\in G_a}\cc{\chi(g)}
 = \frac{1}{\abs{G_a}} \sum_{g\in G_a}\cc{\chi(g)}\;.
\end{equation} 
Since we have at the same time $\alpha_i\in\N_0$ and $\cc{\chi(g)}\in[-d,d]$, so
that $\alpha_i\in[-d,d]$, we conclude that
necessarily $\alpha_i\in\set{0,1,\dots d}$. 
\end{proof}

\begin{proof}[Proof of Lemma~\ref{lem_dim1}]
In the particular case $d=1$, Lemma~\ref{lem_dimd} reduces to 
\begin{equation}
 \label{eq:pgt03}
 \alpha_i = 
 \frac{1}{\abs{G_a}} \sum_{g\in G_a}\cc{\chi(g)} \in\set{0,1}\;.
\end{equation} 
This leaves only two possibilities: either $\alpha_i=1$ and all $\chi(g)=\pi(g)$
are equal to $1$ for $g\in G_a$, or $\alpha_i=0$ and the above sum is equal to
$0$. 
\end{proof}

Note that in the particular case of the trivial representation $\pi=\pi^{(0)}$,
we are always in the case $\alpha_i=1$. Thus all orbits are active for the
trivial representation. 

We now proceed to construct basis vectors for $P_i\C^n$. 
Let $e^a$ denote the canonical basis vector of $\C^n$ associated with
$a\in A$, and let $u^a\in\image \PA$ be defined by  
\begin{equation}
 \label{eq:em04} 
 u^a = \frac{\abs{G}}{\abs{G_a}} \PA  e^a\;.  
\end{equation}
By~\eqref{eq:pgt00a}, its nonzero components are given by 
\begin{equation}
 \label{eq:em04b}
 (u^a)_{h(a)} = \frac{\abs{G}}{\abs{G_a}}(\PA)_{a h(a)} 
 = \frac{d}{\abs{G_a}} \sum_{g\in G_a} \cc{\chi(gh)}\;. 
\end{equation}
This expression is equivalent to~\eqref{sym:14}. For one-dimensional
representations, it reduces to~\eqref{sym:10}. Indeed 
$\chi(gh)=\pi(gh)=\pi(g)\pi(h)$ in dimension $1$, and we can apply
Lemma~\ref{lem_dim1}. For the trivial representation, $(u^a)_{h(a)}$ is
identically equal to $1$.  

In order to compute matrix elements of $L$, we introduce the inner product on
$\C^n$ 
\begin{equation}
 \label{eq:em03}
 \pscal{u}{v} = \frac{1}{\abs{G}} \sum_{g\in G} \cc{u_{g(a)}} v_{g(a)} 
 = \frac{\abs{G_a}}{\abs{G}} \sum_{g\in G/G_a} \cc{u_{g(a)}} v_{g(a)}\;,
\end{equation} 
where $g\in G/G_a$ is a slight abuse of notation for $gG_a\in G/G_a$ (it means
that we pick one representative for each coset $gG_a$). Strictly
speaking, only the restriction of $\pscal{\cdot}{\cdot}$ to the orbit $A$ is an
inner product, since it is not positive definite on all of $\C^n$. 

\begin{lemma}
The vector $u^a$ is normalised in such a way that $\pscal{u^a}{u^a} = \alpha_i
d$. Furthermore, for $v^b$ defined in an analogous way,
\begin{equation}
 \label{eq:em05} 
 \frac{\pscal{u^a}{Lv^b}}{\pscal{u^a}{u^a}} = 
 \frac{d}{\alpha_i\abs{G}\abs{G_b}}
 \sum_{g\in G} \sum_{g'\in G} 
 \cc{\chi(g)}\chi(g') L_{g(a) g'(b)}\;.
\end{equation} 
\end{lemma}
\begin{proof}
We start by computing the norm of $u^a$:
\begin{align}
\nonumber
\pscal{u^a}{u^a} &= \frac{\abs{G}^2}{\abs{G_a}^2} \pscal{\PA e^a}{\PA e^a} 
= \frac{\abs{G}^2}{\abs{G_a}^2} \pscal{e^a}{{\PA}^*\PA  e^a}
= \frac{\abs{G}^2}{\abs{G_a}^2} \pscal{e^a}{\PA  e^a} \\
&= \frac{\abs{G}}{\abs{G_a}} \sum_{g\in G/G_a} \cc{e^a_{g(a)}} (\PA e^a)_{g(a)} 
= \frac{\abs{G}}{\abs{G_a}} (\PA )_{a a} 
= \frac{d}{\abs{G_a}} \sum_{h\in G_a} \cc{\chi(h)}
= \alpha_i d\;,
\end{align}
where we have used the fact that $\PA $ is a hermitian projector. 
Before turning to the numerator of~\eqref{eq:em05}, note that for a matrix 
$M\in\C^{n\times n}$ one has 
\begin{equation}
 \pscal{e^a}{Me^b} = \frac{\abs{G_a}}{\abs{G}}M_{ab}\;.
\end{equation} 
Therefore,
\begin{align}
\nonumber
\pscal{u^a}{Lv^b} 
&= \frac{\abs{G}^2}{\abs{G_a}\abs{G_b}} \pscal{\PA e^a}{L\PB e^b}
= \frac{\abs{G}^2}{\abs{G_a}\abs{G_b}} \pscal{e^a}{\PA L\PB e^b}\\
&= \frac{\abs{G}}{\abs{G_b}} (\PA L\PB )_{ab}\;.
\end{align}
Now we have 
\begin{align}
\nonumber
(\PA L\PB )_{ab}
&= \sum_{g\in G/G_a} \sum_{g'\in G/G_b} (\PA )_{a g(a)} L_{g(a) g'(b)}
(\PB )_{g'(b) b} \\
&= \frac{d^2}{\abs{G}^2} 
\sum_{g\in G/G_a} \sum_{g'\in G/G_b} \sum_{h\in G_a} \sum_{h'\in G_b} 
\cc{\chi(gh)} \; \cc{\chi((g')^{-1}h')} L_{g(a) g'(b)}\;.
\end{align}
Since $\cc{\chi((g')^{-1}h')}=\chi(g'(h')^{-1})$, the result follows by 
replacing first $(h')^{-1}$ by $h'$ in the sum, and then $gh$ by $g$ and $g'h'$
by $g'$.  
\end{proof}

The expression~\eqref{eq:em05} for the matrix elements can be simplified with
the help of the following identity. 

\begin{lemma}
For any $h\in G$,  
\begin{equation}
 \label{eq:em06}
 \frac{d}{\abs{G}} \sum_{g\in G} \cc{\chi(g)} \chi(hg) = \chi(h)\;.
\end{equation} 
\end{lemma}
\begin{proof}
Let $G$ act on a set $X$ such that $G_x=\set{\id}$ for all $x\in X$.
By~\eqref{eq:pgt00a} we have
\begin{equation}
 (P_X)_{x,h(x)} = \frac{d}{\abs{G}} \cc{\chi(h)}\;.
\end{equation} 
This implies 
\begin{equation}
 (P_X^2)_{x h(x)} = \sum_{g\in G} (P_X)_{x g(x)}(P_X)_{g(x) h(x)} 
 = \frac{d^2}{\abs{G}^2} \sum_{g\in G} \cc{\chi(g)}\chi(h^{-1}g)\;.
\end{equation} 
Since $P_X$ is a projector, the two above expressions are equal. 
\end{proof}

\begin{cor}
The expressions~\eqref{eq:em05} of the matrix elements simplify to  
\begin{equation}
 \label{eq:em07} 
 \frac{\pscal{u^a}{Lv^b}}{\pscal{u^a}{u^a}} = 
 \frac{1}{\alpha_i\abs{G_b}}
 \sum_{g\in G} 
 \chi(g) L_{a g(b)}\;.
\end{equation} 
\end{cor}
\begin{proof}
This follows by setting $g'=gh$ in~\eqref{eq:em05}, using $L_{g(a)gh(b)} =
L_{ah(b)}$ and applying the lemma. This is possible since
$\chi(gh)=\Tr(\pi(g)\pi(h))=\Tr(\pi(h)\pi(g))=\chi(hg)$. 
\end{proof}

By the non-degeneracy Assumption~\ref{assump_sym_nondeg}, the sum
in~\eqref{eq:em07} will be dominated by a few terms only. Using this, 
general matrix elements of $L$ can be rewritten as follows. 

\begin{prop}
Let $A$ and $B$ be two different orbits, and assume that $a\in A$ and $b\in B$
are such that $h_{ab}=h^*(A,B)$, the minimal exponent for transitions from $A$
to $B$. Then for any $h_1, h_2\in G$, 
\begin{equation}
 \label{eq:em08} 
  \frac{\pscal{u^{h_1(a)}}{Lv^{h_2(b)}}}{\pscal{u^{h_1(a)}}{u^{h_1(a)}}} = 
 \frac{L_{ab}}{\alpha_i\abs{G_b}} \sum_{g\in G_aG_b} \chi(h_1 g h_2^{-1})
  [1+\Order{\e^{-\theta/\eps}}]\;.
\end{equation} 
Furthermore, elements of diagonal blocks can be written as 
\begin{equation}
 \label{eq:em10} 
  \frac{\pscal{u^{h_1(a)}}{Lu^{h_2(a)}}}{\pscal{u^{h_1(a)}}{u^{h_1(a)}}} = 
 \frac{1}{\alpha_i\abs{G_a}} \sum_{g\in G_a} 
  \biggbrak{\chi(h_1 g h_2^{-1})L_{aa} + \sum_{k\in G/G_a\setminus G_a}
\chi(h_1kgh_2^{-1})L_{ak(a)}}\;.
\end{equation} 
\end{prop}

\begin{proof}
It follows from~\eqref{eq:em07} that 
\begin{align}
\nonumber
\frac{\pscal{u^{h_1(a)}}{Lv^{h_2(b)}}}{\pscal{u^{h_1(a)}}{u^{h_1(a)}}} 
&= \frac{1}{\alpha_i\abs{G_{h_2(b)}}}\sum_{g\in G} 
 \chi(g) 
 \underbrace{L_{h_1(a) \, gh_2(b)}}_{=L_{a \, h_1^{-1}gh_2(b)}} \\
&= \frac{1}{\alpha_i\abs{G_b}} \sum_{k\in G_aG_b} 
 \chi(h_1 k h_2^{-1}) L_{ab}[1+\Order{\e^{-\theta/\eps}}] \;,
\end{align}
where we have set $k=h_1^{-1}gh_2$, and used~\eqref{eq:sym04a} and
Lemma~\ref{lem_GaGb}. This proves~\eqref{eq:em08}. Relation~\eqref{eq:em10}
follows from~\eqref{eq:em07} after replacing $g\in G$ by $kg$, with $g\in
G_a$ and $k\in G/G_a$, and singling out the term $k=\id$.  
\end{proof}

Expression~\eqref{eq:em08} is equivalent to~\eqref{eq:sym15} in
Proposition~\ref{prop_matrix_Lp}, taking into account the
definition~\eqref{eq:sym10} of $c^*_{ij}$ and $m^*_i$. Particularising to
one-dimensional representations yields~\eqref{eq:sym11} and~\eqref{eq:sym09}. 

It thus remains to determine the diagonal blocks. For
one-dimensional representations, using $\chi(kg)=\pi(kg)=\pi(k)\pi(g)$ and
Lemma~\ref{lem_dim1} in~\eqref{eq:em10} shows that 
\begin{equation}
 \label{eq:em11}
L^{(p)}_{ii} := 
\frac{\pscal{u^a}{Lu^a}}{\pscal{u^a}{u^a}} 
= L_{aa} + \sum_{k\in G/G_a\setminus G_a} \pi(k)L_{ak(a)}\;.
\end{equation} 
Subtracting $L^{(0)}$ for the trivial representation from
$L^{(p)}$ proves~\eqref{eq:sym13}. Furthermore, let $\vone$ be the constant
vector with all components equal to $1$. Since $L$ is a generator, we have 
\begin{equation}
 \label{eq:em11b}
 0 = L \vone = \sum_{j=1}^m Lu_j^{(0)} 
 \qquad \Rightarrow \qquad
 0 = \sum_{j=1}^m L_{ij}^{(0)}\;,
\end{equation} 
which proves \eqref{eq:sym10b}. 

Finally let $a^*$ be such that $h(a,a^*)=\inf_b
h(a,b)$. We distinguish two cases:

\begin{enum}
\item 	Case $a^*\notin A$. Then the right-hand side of~\eqref{eq:em11} is
dominated by the first term, and we have
$L^{(p)}_{ii}=L_{aa}[1+\Order{\e^{-\theta/\eps}}]$. The sum inside the brackets
in~\eqref{eq:em10} is also dominated by the first term, which implies the first
lines in~\eqref{eq:sym17} and in~\eqref{eq:sym14}. 

\item 	Case $a^*=k_0(a)\in A$. Relation~\eqref{eq:sym10d} implies that
$L_{ii}^{(0)}$ is negligible with respect to $L_{aa^*}$, and thus  
\begin{equation}
 \label{eq:em12}
 L_{aa} = - \sum_{k\in G/G_a\setminus G_a} L_{ak(a)}
[1+\Order{\e^{-\theta/\eps}}]\;.
\end{equation} 
Thus for one-dimensional representations, we obtain from~\eqref{eq:em11} that 
\begin{equation}
 \label{eq:em13}
L^{(p)}_{ii} := 
\frac{\pscal{u^a}{Lu^a}}{\pscal{u^a}{u^a}} 
= - \sum_{k\in G/G_a\setminus G_a}
(1-\pi(k))L_{ak(a)}[1+\Order{\e^{-\theta/\eps}}]\;.
\end{equation} 
The sum is dominated by $k=k_0$ and $k=k_0^{-1}$, 
which implies the last two lines of~\eqref{eq:sym14}.  
For general representations, we obtain from~\eqref{eq:em10} that 
\begin{equation}
 \label{eq:em14} 
  \frac{\pscal{u^{h_1(a)}}{Lu^{h_2(a)}}}{\pscal{u^{h_1(a)}}{u^{h_1(a)}}} = 
 -\frac{1+\Order{\e^{-\theta/\eps}}}{\alpha_i\abs{G_a}} \sum_{g\in G_a} 
  \sum_{k\in G/G_a\setminus G_a}
\bigbrak{\chi(h_1 g h_2^{-1}) - \chi(h_1kgh_2^{-1})}L_{ak(a)}\;,
\end{equation} 
which implies the second line in~\eqref{eq:sym17}.
\end{enum}

%%%%%%%%%%%%%%%%%%%%%%%%%%%%%%%%%%%%%%%%%%%%%%%%%%%%%%%%%%%%%%%%%%%%%%%%%%%%%%

\section{Proofs -- Estimating eigenvalues}
\label{sec_proofT}

%%%%%%%%%%%%%%%%%%%%%%%%%%%%%%%%%%%%%%%%%%%%%%%%%%%%%%%%%%%%%%%%%%%%%%%%%%%%%%

\subsection{Block-triangularisation}
\label{ssec_triang} 

We consider in this section a generator $L\in\R^{n\times n}$ with matrix
elements $L_{ij}=\e^{-h_{ij}/\eps}$, satisfying
Assumption~\ref{assump_asym_meta}
on existence of a metastable hierarchy. In this section, we have incorporated
the prefactors in the exponent, i.e., we write $h_{ij}$ instead of $h_{ij} -
\eps\log(c_{ij}/m_i)$ and $V_i$ instead of $V_i+\eps\log(m_i)$. 

In addition, we assume the \emph{reversibility condition for minimal paths} 
\begin{equation}
 \label{eq:bti01}
 V_i + H(i,j) = V_j + H(j,i) + \Order{\eps\e^{-\theta/\eps}}
 \qquad
 \forall i, j \in \set{2,\dots,n}\;.
\end{equation} 
If the reversibility assumption~\eqref{meta02a} holds, then~\eqref{eq:bti01} is
satisfied. However~\eqref{eq:bti01} is slightly weaker, because it only concerns
minimal transition paths. We do not assume reversibility for site $1$, since
this will allow us to cover situations associated with nontrivial
representations. Thus the first row of $L$ may be identically zero, making $1$
an absorbing state. 

Our aim is to construct a linear change of variables transforming $L$ into a
triangular matrix. The change of variables is obtained by combining $n-1$
elementary transformations to block-triangular form. 
Given some $1\leqs m<n$, we write $L$ in the form 
\begin{equation}
 \label{eq:bti02} 
 L=\begin{pmatrix}
L^{11}&L^{12} \\
L^{21}&A
\end{pmatrix}\;,
\end{equation} 
with blocks $L^{11}\in \R^{(n-m)\times (n-m)}$, 
$A\in \R^{m\times m}$, 
$L^{12}\in \R^{(n-m)\times m}$ and 
$L^{21}\in \R^{m\times (n-m)}$, and we assume 
$\det(A)\neq 0$.
We would like to construct matrices $S, T\in\R^{n\times n}$ satisfying 
\begin{equation}
 \label{eq:bti03}
 LS=ST
\end{equation} 
where $T$ is block-triangular. More precisely, we impose that 
\begin{equation}
 \label{eq:bti04}
S=\begin{pmatrix}
\one&S^{12} \\
0&\one
\end{pmatrix}
\qquad \text{and} \qquad
T=\begin{pmatrix}
T^{11}&0 \\
T^{21}&\tilde{A}
\end{pmatrix}\;,
\end{equation} 
with blocks of the same dimensions as the blocks of $L$.
Plugging~\eqref{eq:bti04} into~\eqref{eq:bti03} yields the relations 
\begin{align}
\nonumber
T^{11} &= L^{11} - S^{12} L^{21}\;, \\
\label{eq:bti05} 
\tilde{A} &= A + L^{21}S^{12}\;, \\
T^{21} &= L^{21}\;, 
\nonumber
\end{align}
and 
\begin{equation}
 \label{eq:bti06} 
 L^{11}S^{12} - S^{12}A - S^{12}L^{21}S^{12} + L^{12} = 0\;.
\end{equation} 
If we manage to prove that~\eqref{eq:bti06} admits a solution, then we will
have shown that $L$ is similar to the block-diagonal matrix $T$, and the
eigenvalues of $L$ are those of $T^{11}$ and $\tilde A$.  
In the sequel, the size of matrices is measured in the operator sup-norm, 
\begin{equation}
 \label{eq:bti07} 
 \norm{L} = \sup_{\norm{x}_\infty=1} \norm{Lx}_\infty\;,
 \qquad
 \norm{x}_\infty = \sup_i\abs{x_i}\;.
\end{equation}

\begin{prop}
\label{prop:S12} 
If $\norm{L^{12}A^{-1}}$ is sufficiently small, then~\eqref{eq:bti06} admits a
solution $S^{12}$, such that  $\norm{S^{12}}=\Order{\norm{L^{12}A^{-1}}}$. 
\end{prop}

\begin{proof}
For fixed blocks $A, L^{21}$, consider the function
\begin{align}
\nonumber 
f :\R^{(n-m)\times m} \times \R^{(n-m)\times n} &\to \R^{(n-m)\times m} \\
 (X, (L^{11},L^{12})) &\mapsto L^{11}XA^{-1} - X - XL^{21}XA^{-1} +
L^{12}A^{-1}\;.
 \label{eq:bti08}
\end{align} 
Then $f(0,0)=0$, and the Fr\'echet derivative of $f$ with respect to $X$ at
$(0,0)$ is given by $\partial_Xf(0,0)=-\id$. Hence the implicit-function theorem
applies, and shows the existence of a map $X^*:\R^{(n-m)\times
n}\to\R^{(n-m)\times m}$ such that $f(X^*, (L^{11},L^{12}))=0$ in a
neighbourhood of $(0,0)$. Then $S^{12}=X^*(L^{11},L^{12})$
solves~\eqref{eq:bti06}. Furthermore,
$\norm{S^{12}}=\Order{\norm{L^{12}A^{-1}}}$ follows from the expression for the
derivative of the implicit function. 
\end{proof}

The first-order Taylor expansion of $S^{12}$ reads 
\begin{equation}
 \label{eq:bti09}
 S^{12} = L^{12}A^{-1} + 
 \BigOrder{\norm{L^{12}A^{-1}} \bigbrak{\norm{L^{11}A^{-1}} +
\norm{L^{21}L^{12}A^{-2}}}}\;.
\end{equation} 
We will start by analysing the first-order approximation obtained by using
$S_0^{12} = L^{12}A^{-1}$. The resulting transformed matrix is 
\begin{equation}
 \label{eq:bti10}
T_0 = 
\begin{pmatrix}
T_0^{11} & 0 \\
L^{21} & \tilde A_0 
\end{pmatrix}
= 
\begin{pmatrix}
L^{11} - L^{12}A^{-1}L^{21} & 0 \\
L^{21} & A + L^{21}L^{12}A^{-1}
\end{pmatrix}
\end{equation} 

\begin{lemma}
The matrix  $T_0^{11}$ is still a generator.
\end{lemma}

\begin{proof}
The fact that $L$ is a generator implies $L^{11}\vone + L^{12}\vone = 0$ 
and $L^{21}\vone + A \vone = 0$, where $\vone$ denotes the constant vector of
the appropriate size. It follows that 
\begin{equation}
 \label{eq:bti11} 
L^{12}A^{-1}L^{21}\vone = L^{12} A^{-1} (-A\vone) = -L^{12}\vone =
L^{11}\vone\;,
\end{equation}
and thus $T_0^{11}\vone = 0$. 
\end{proof}

We will see that $T_0^{11}$ can be interpreted as the generator of a jump
process in which the sites $i>n-m$ have been \lq\lq erased\rq\rq. Our strategy
will be to show that this reduced process has the same communication heights as
the original one, and then to prove that higher-order terms in the expansion of
$S^{12}$ do not change this fact. We can then apply the same strategy to the
block $T^{11}$, and so on until the resulting matrix is block-triangular with
blocks of size $m$. The diagonal blocks of this matrix then provide the
eigenvalues of~$L$.

\subsection{The one-dimensional case}

We consider in this section the case $m=1$, which allows to cover all
one-dimensional representations. The lower-right block $A$ of $L$ is then a
real number that we denote $a$~$(=L_{nn})$, and we write $\tilde a$ instead of
$\tilde A$. 

\subsubsection*{The first-order approximation}

The matrix elements of $T_0^{11}$ are given by (c.f.~\eqref{eq:bti10}) 
\begin{equation}
 \label{eq:b1d01}
 T^0_{ij} = L_{ij} - \frac1a L_{in}L_{nj}\;, 
 \qquad
 i, j = 1, \dots n-1\;.
\end{equation} 
Assumption~\ref{assump_asym_meta} implies that there is a unique successor $k =
s(n) \in \set{1,n-1}$ such that $h_{nk}=\min_{j\in\set{1,n-1}} h_{nj}$.
Since $L$ is a generator, we have $a = -\e^{-h_{nk}/\eps}
[1+\Order{\e^{-\theta/\eps}}]$, and thus $T^0_{ij} = \e^{-\tilde
h_{ij}/\eps}$ where 
\begin{equation}
 \label{eq:b1d02}
 \tilde h_{ij} = \tilde h^0_{ij} + \Order{\eps\e^{-\theta/\eps}}
 \qquad 
 \text{with}
 \qquad
 \tilde h^0_{ij} = h_{ij} \wedge (h_{in}-h_{nk}+h_{nj})\;.
\end{equation} 
The new exponent $\tilde h^0_{ij}$ can be interpreted as the lowest cost to go
from site $i$ to site $j$, possibly visiting $n$ in between. 

We denote by $\widetilde H^0(i,j)$ the new communication height between sites
$i,j\in\set{1,\dots,n-1}$, defined in the same way as $H(i,j)$ but using $\tilde
h^0_{ij}$ instead of $h_{ij}$ ($p$-step communication heights are defined
analogously). In order to show that the new communication heights are in fact
equal to the old ones, we start by establishing a lower bound. 

\begin{lemma}
\label{lem_b1d2} 
For all $i\neq j\in\set{1,\dots,n-1}$, 
\begin{equation}
 \label{eq:b1d03}
 \tilde h^0_{ij} \geqs h_{ij} \wedge h_{inj}\;,
\end{equation} 
with equality holding if $i=k$ or $j=k$. As a consequence,
$\widetilde H^0(i,j) \geqs H(i,j)$ for these $i, j$. 
\end{lemma}

\begin{proof}
Recall from Definition~\ref{def_comm_heights} that the two-step communication
height $i \to n \to j$ is given by 
$h_{inj} = h_{in}\vee(h_{in}-h_{ni}+h_{nj})$. 
We consider three cases:
\begin{itemiz}
\item 	If $i=k$, then $h_{knj} = h_{kn}-h_{nk}+h_{nj}$ because $h_{nj}>h_{nk}$,
and thus $\tilde h^0_{kj} = h_{kj}\wedge h_{knj}$. 

\item 	If $j=k$, then $h_{ink} = h_{in}$ because $h_{nk}<h_{ni}$,
and thus $\tilde h^0_{ik} = h_{ik}\wedge h_{in} = h_{ik}\wedge h_{ink}$. 

\item 	If $i\neq k\neq j$, then $h_{in} - h_{nk} + h_{nj} > h_{inj}$
because $h_{nk} < h_{ni}, h_{nj}$ and \eqref{eq:b1d03} holds.
\end{itemiz}
The consequence on communication heights follows by comparing maximal
heights along paths from $i$ to $j$. 
\end{proof}

\begin{figure}
\begin{center}
\begin{tikzpicture}[->,>=stealth',shorten >=1pt,auto,node distance=2.5cm,
  thick,main
node/.style={circle,minimum size=0.7cm,fill=blue!20,draw}]
\node at (-1.0,2.8) {Rule 1};
\node[main node] (n) at (0,0) {$n$};
\node[main node] (i) at (-1.2,1.0) {$i$};
\node[main node] (j) at (1.2,1.0) {$j$};
\draw (i)  -- node[below] {$h_{ij}$} (j);
\draw (i) edge[color=red,bend left] node {$\tilde h^0_{ij}$} (j);
\end{tikzpicture}
\hspace{10mm}
\begin{tikzpicture}[->,>=stealth',shorten >=1pt,auto,node distance=2.5cm,
  thick,main
node/.style={circle,minimum size=0.7cm,fill=blue!20,draw}]
\node at (0.1,2.8) {Rule 2};
\node[main node] (n) at (0,0) {$n$};
\node[main node] (k) at (0.0,2.0) {$k$};
\node[main node] (j) at (1.2,1.0) {$j$};
\draw (k) --(n);
\draw (n) --(j);
\draw (k) edge[color=red,bend left] node {$\tilde h^0_{kj}$} (j);
\end{tikzpicture}
\hspace{10mm}
\begin{tikzpicture}[->,>=stealth',shorten >=1pt,auto,node distance=2.5cm,
  thick,main
node/.style={circle,minimum size=0.7cm,fill=blue!20,draw}]
\node at (-1.0,2.8) {Rule 3};
\node[main node] (n) at (0,0) {$n$};
\node[main node] (k) at (0.0,2.0) {$k$};
\node[main node] (i) at (-1.2,1.0) {$i$};
\draw (i) --(n);
\draw (n) --(k);
\draw (i) edge[color=red,bend left] node {$\tilde h^0_{ik}$} (k);
\end{tikzpicture}
\hspace{10mm}
\begin{tikzpicture}[->,>=stealth',shorten >=1pt,auto,node distance=2.5cm,
  thick,main
node/.style={circle,minimum size=0.7cm,fill=blue!20,draw}]
\node at (-1.0,2.8) {Rule 4};
\node[main node] (n) at (0,0) {$n$};
\node[main node] (k) at (0.0,2.0) {$k$};
\node[main node] (i) at (-1.2,1.0) {$i$};
\node[main node] (j) at (1.2,1.0) {$j$};
\draw (i) --(n);
\draw (n) --(j);
\draw[dashed] (n) --(k);
\draw[color=red] (i) to[out=110,in=150] node {$\gamma_1$} (k);
\draw[color=red] (k) to[out=30,in=70] node {$\gamma_2$} (j);
\end{tikzpicture}
\end{center}
\vspace{-3mm}
\caption[]{Replacement rules for minimal paths.}
\label{fig:minimal_paths} 
\end{figure}
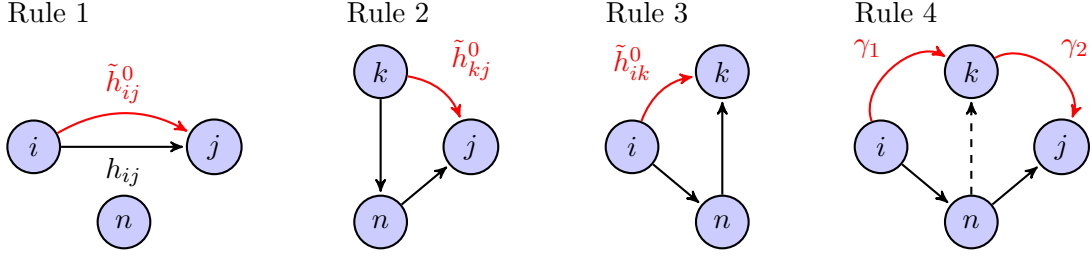

\begin{prop}
\label{prop:b1d1} 
For all $i\neq j\in\set{1,\dots,n-1}$ and sufficiently small $\eps$, 
\begin{equation}
 \label{eq:b1d04}
 \widetilde H^0(i,j) = H(i,j) + \Order{\eps\e^{-\theta/\eps}}\;. 
\end{equation} 
\end{prop}
\begin{proof}
Let $\gamma$ be a minimal path between two sites $i_0$ and $j_0$. In view of
Lemma~\ref{lem_b1d2}, it is sufficient to construct a path $\tilde\gamma$ from 
$i_0$ to $j_0$, which does not include $n$, such that $\tilde
h^0_{\tilde\gamma} = h_\gamma$. This new path is obtained by applying the
following replacement rules (see \figref{fig:minimal_paths}):
\begin{enum}
\item 	leave as is each segment $i\to j$ with $i,j \neq k, n$;
\item 	replace any segment $k\to n\to j$ with $j\neq k, n$ by $k\to j$;
\item 	replace any segment $i\to n\to k$ with $i\neq k, n$ by $i\to k$;
\item 	replace any segment $i\to n\to j$ with $i, j\neq k, n$ by the
concatenation of a minimal path $\gamma_1:i\to k$ and a minimal path 
$\gamma_2:k\to j$. If one of these paths contains $n$, apply rules 2.\ or 3.
\end{enum}
It is sufficient to show that each of these modifications leaves invariant the
local communication height. 
\begin{enum}
\item 	Segment $i\to j$ with $i,j \neq k, n$: $h_{ij}\leqs h_{inj}$ because the
path is minimal; thus either $h_{ij} \leqs h_{in}$ and thus $\tilde h^0_{ij}
= h_{ij}\wedge(h_{in}-h_{nk}+h_{nj})=h_{ij}$ because $h_{nk} < h_{nj}$. Or 
$h_{ij} \leqs h_{in}-h_{ni}+h_{nj} < h_{in}-h_{nk}+h_{nj}$ and thus again 
$\tilde h^0_{ij} =h_{ij}$.

\item 	Segment $k\to n\to j$ with $j\neq k, n$: 
Then $\tilde h^0_{kj} = h_{kj} \wedge (h_{kn}-h_{nk}+h_{nj}) =
h_{kn}-h_{nk}+h_{nj}$ because the path $k\to n\to j$ is minimal, and we have
seen in the previous lemma that this is equal to $h_{knj}$. Thus $\tilde
h^0_{kj} =h_{knj}$.

\item 	Segment $i\to n\to k$ with $i\neq k, n$:
Here $\tilde h^0_{ik} = h_{ik} \wedge h_{in}$. We have seen in the previous
lemma that $h_{in}=h_{ink}$, which must be smaller than $h_{ik}$ because the
path is minimal. We conclude that $\tilde h^0_{ik} =h_{ink}$.

\item 	Segment $i\to n\to j$ with $i, j\neq k, n$:
In this case we have $\tilde h^0_{ikj} = h_{inj} + 
\Order{\eps\e^{-\theta/\eps}}$. Indeed, 
\begin{itemiz}
\item 	By minimality of the path, 
$h_{in} \leqs H(i,k)\vee(H(i,k)-H(k,i)+h_{kn})$.
The reversibility assumption~\eqref{eq:bti01} 
and the minimality of $\gamma_1$ and $n\to k$ yield  
\begin{align}
 \nonumber 
 H(i,k)-H(k,i)+h_{kn} 
 &= V_n - V_i + h_{nk}  + \Order{\eps\e^{-\theta/\eps}} \\
 \nonumber 
 &= h_{in}-h_{ni}+h_{nk}  + \Order{\eps\e^{-\theta/\eps}} \\
 &< h_{in}  + \Order{\eps\e^{-\theta/\eps}}
 \label{eq:b1d05a} 
\end{align}
and thus $h_{in} \leqs H(i,k) \leqs h_{ik}$ for sufficiently small $\eps$. This
implies 
\begin{equation}
 \label{eq:b1d05}
 \tilde h^0_{ik} = h_{ik} \wedge h_{in} = h_{in}\;.
\end{equation} 

\item 	Minimality also yields $h_{nj} \leqs  
h_{nk} \vee (h_{nk} - h_{kn} + H(k,j)) = h_{nk} - h_{kn} + H(k,j)$, where we
have used $h_{nj} > h_{nk}$. Thus $h_{kn}-h_{nk}+h_{nj}\leqs H(k,j)\leqs
h_{kj}$, which implies 
\begin{equation}
 \label{eq:b1d06}
 \tilde h^0_{kj} = h_{kn}-h_{nk}+h_{nj}\;.
\end{equation} 

\item 	By assumption~\eqref{eq:bti01}, 
\begin{align}
\nonumber
h_{kn}-h_{nk}+h_{ni} 
&= H(k,i)-H(i,k)+h_{in}  + \Order{\eps\e^{-\theta/\eps}} \\ 
\nonumber
&\leqs H(k,i) + \Order{\eps\e^{-\theta/\eps}} \\
&\leqs h_{ki} + \Order{\eps\e^{-\theta/\eps}}\;, 
 \label{eq:b1d06a}
\end{align}
since $h_{in}\leqs H(i,k)$, so that 
\begin{equation}
 \label{eq:b1d07}
 \tilde h^0_{ki} = h_{kn}-h_{nk}+h_{ni} + \Order{\eps\e^{-\theta/\eps}}\;.
\end{equation} 
\end{itemiz}
Combining~\eqref{eq:b1d05}, \eqref{eq:b1d06} and~\eqref{eq:b1d07}, we obtain 
$\tilde h^0_{ikj} = h_{inj} + \Order{\eps\e^{-\theta/\eps}}$, concluding the
proof. 
\qed
\end{enum}
\renewcommand{\qed}{}
\end{proof}

\subsubsection*{The full expansion}

It remains to extend the previous results from the first-order approximation
$S^{12}_0$ to the exact solution $S^{12}$. 

\begin{prop}
\label{prop:b1d2} 
For sufficiently small $\eps$, the matrix $S^{12}$ satisfying~\eqref{eq:bti06}
is given by the convergent series 
\begin{equation}
\label{eq:b1d08} 
 S^{12} = \sum_{p=0}^\infty \frac{1}{a^{p+1}} \bigpar{L^{11}}^p L^{12} 
 \bigbrak{1+\Order{\e^{-\theta/\eps}}}\;.
\end{equation} 
\end{prop}
\begin{proof}
First observe that by Assumption~\ref{assump_asym_meta}, $\norm{L^{21}A^{-1}} =
\abs{a}^{-1}\norm{L^{21}} = \Order{\e^{-\theta/\eps}}$. Thus by
Proposition~\ref{prop:S12}, \eqref{eq:bti06} admits a solution $S^{12}$ of
order $\e^{-\theta/\eps}$. This solution satisfies 
\begin{equation}
 \label{eq:b1d09} 
 S^{12} = \frac1a L^{12} + \frac1a L^{11}S^{12} - 
 \frac{S^{12}L^{21}}{a} S^{12}\;.
\end{equation} 
Note that $S^{12}L^{21}/a$ is a scalar of order $\e^{-\theta/\eps}$. It follows
that 
\begin{equation}
 \label{eq:b1d10}
 S^{12} = \frac1a \Bigbrak{\one - \frac1a L^{11}}^{-1} L^{12} 
 \bigbrak{1+\Order{\e^{-\theta/\eps}}}\;,
\end{equation} 
and the conclusion follows by writing the inverse as a geometric series. 
\end{proof}

Plugging~\eqref{eq:b1d08} into~\eqref{eq:bti05}, we obtain 
\begin{equation}
 \label{eq:b1d11}
 T^{11} = L^{11} + \sum_{p=0}^\infty \frac{1}{a^{p+1}} \bigpar{L^{11}}^p L^{12}
L^{21} \bigbrak{1+\Order{\e^{-\theta/\eps}}}\;.
\end{equation} 
$L^{11}$ and the term $p=0$ correspond to the first-order approximation
$T^{11}_0$. It follows that the matrix elements of $T^{11}$ are of the form
$\smash{\e^{-\tilde h_{ij}/\eps}}$ where
\begin{equation}
 \label{eq:b1d12}
 \tilde h_{ij} = \tilde h^0_{ij} \wedge 
 \inf_{\substack{p\geqs 1 \\ 1\leqs l_1,\dots, l_p\leqs{n-1}}}
 \Bigpar{h_{il_1}+h_{l_1l_2}+\dots+h_{l_pn}+h_{nj} - (p+1)h_{nk}}
 + \Order{\eps\e^{-\theta/\eps}}\;.
\end{equation}
In order to control the remainder terms, we establish the following estimate.

\begin{lemma}
\label{lem_b1d3}
For any $p\geqs 1$, any $i, l_1, \dots, l_p \in\set{1,\dots,n-1}$ and 
$j\in\set{1,\dots,n}$,
\begin{align}
 \label{eq:b1d13a}
 h_{il_1} + h_{l_1l_2} + \dots + h_{l_pj} - p h_{nk} 
 &\geqs h_{il_1\dots l_pj} + p\theta\;, \\
 h_{il_1} + h_{l_1l_2} + \dots + h_{l_pn} + h_{nj} - (p+1) h_{nk} 
 &\geqs h_{il_1\dots l_pnj} + p\theta\;.
 \label{eq:b1d13b}
\end{align} 
\end{lemma}
\begin{proof}
We prove first~\eqref{eq:b1d13a} for $p=1$. If
$h_{il}>h_{il}-h_{li}+h_{lj}$ then $h_{ilj}=h_{il}$. This implies 
$h_{il}+h_{lj}-h_{nk} = h_{ilj} + (h_{lj}-h_{nk}) \geqs h_{ilj}+\theta$, where
we have used~\eqref{eq:asym04}. Otherwise $h_{ilj} = h_{il}-h_{li}+h_{lj}$, and
then $h_{il}+h_{lj}-h_{nk} = h_{ilj} + (h_{li}-h_{nk}) \geqs h_{ilj}+\theta$. 
The proof easily extends by induction to general $p$, using the
definition~\eqref{eq:asym01} of communication heights and the fact that
$h_{ij}-h_{nk} \geqs \theta$ for $i=1,\dots n-1$. 

To prove the second inequality~\eqref{eq:b1d13b} for $p=1$, we use that if 
$h_{iln} \geqs h_{il}-h_{li}+h_{ln}-h_{nl}+h_{nj}$, then $h_{ilnj} = h_{iln}$
and thus $h_{il}+h_{ln}+h_{nj}-2h_{nk} = (h_{il}+h_{ln}-h_{nk})+
(h_{nj}-h_{nk})$ so that the conclusion follows from~\eqref{eq:b1d13a} and the
fact that $h_{nj}\geqs h_{nk}$. Otherwise we have 
$h_{ilnj} = h_{il}-h_{li}+h_{ln}-h_{nl}+h_{nj}$ and  
$h_{il}+h_{ln}+h_{nj}-2h_{nk} = h_{ilnj}+(h_{li}-h_{nk})+(h_{nl}-h_{nk})$,
which is greater or equal $h_{ilnj} + \theta$. The proof then extends by
induction to general $p$. 
\end{proof}

\begin{cor}
\label{cor_b1d1} 
For all $i\neq j\in\set{0,\dots,n-1}$, 
\begin{equation}
 \label{eq:b1d14}
 \tilde h_{ij} = \tilde h^0_{ij} \wedge R_{ij} + \Order{\eps\e^{-\theta/\eps}}
 \qquad 
 \text{where }
 R_{ij} \geqs H(i,j) + \theta\;.
\end{equation} 
\end{cor}
\begin{proof}
This follows directly from~\eqref{eq:b1d12}, \eqref{eq:b1d13b} and the
definition~\eqref{eq:asym02} of the communication height $H(i,j)$. 
\end{proof}

\begin{cor}
\label{cor_b1d2}
Communication heights are preserved to leading order in $\eps$, that is, 
\begin{equation}
 \label{eq:bld15}
 \widetilde H(i,j) = H(i,j) + \Order{\eps\e^{-\theta/\eps}}
 \qquad \forall i,j\in\set{1,\dots,n-1}\;.
\end{equation} 
\end{cor}
\begin{proof}
Corollary~\ref{cor_b1d1} and Proposition~\ref{prop:b1d1} directly yield
$\widetilde H(i,j) \leqs H(i,j) + \Order{\eps\e^{-\theta/\eps}}$ since $\tilde
h_{ij} \leqs \tilde h^0_{ij} + \Order{\eps\e^{-\theta/\eps}}$ implies that
maximal heights encountered along paths do not increase. To show equality,
consider an optimal path $\tilde\gamma:i\to j$. Relation~\eqref{eq:b1d14}
applied to each segment of $\tilde\gamma$ shows that $\gamma$ is also an optimal
path for the original generator. 
\end{proof}

Note that this result shows in particular that assumption~\eqref{eq:b1d01} on
reversibility for optimal paths is satisfied by the new communication heights.
We can now state the main result of this section, which characterises the
eigenvalues of a generator admitting a metastable hierarchy. 

\begin{theorem}[Eigenvalues of a metastable generator]
\label{thm_triang_d1} 
Let $L$ be a generator satisfying Assumption~\ref{assump_asym_meta} on
existence of a metastable hierarchy and the reversibility condition for minimal
paths~\eqref{eq:bti01}. For sufficiently small $\eps$, the eigenvalues of $L$
are given by $\lambda_1=0$ and 
\begin{equation}
 \label{eq:bld16}
 \lambda_k = -%\frac{c_{e^*(k)}}{m_k} 
 \e^{-H(k,\cM_{k-1})/\eps}
 \bigbrak{1+\Order{\e^{-\theta/\eps}}}\;, 
 \qquad
 k=2,\dots,n\;.
\end{equation} 
\end{theorem}
\begin{proof}
Since $L$ is a generator, necessarily $\lambda_1=0$. Furthermore, 
$L$ has the same eigenvalues as 
\begin{equation}
 \label{eq:b1d20}
 T = 
 \begin{pmatrix}
 T^{11} & 0 \\ T^{12} & \tilde a
 \end{pmatrix}\;,
\end{equation} 
where $\tilde a = a + L^{21}S^{12}$. Assumption~\ref{assump_asym_meta} and the
fact that $L$ is a generator imply that 
\begin{equation}
 \label{eq:b1d21} 
 a = \e^{-h_{nk}/\eps} \bigbrak{1+\Order{\e^{-\theta/\eps}}}\;,
\end{equation} 
with $h_{nk} = h_{ns(n)} = H(n,\cM_{n-1})$. Furthermore, we have
$\norm{L^{21}}=\Order{a}$ and Proposition~\ref{prop:S12} shows that 
$\norm{S^{12}}=\Order{\norm{L^{12}a^{-1}}}=\Order{\e^{-\theta/\eps}}$. 
Thus $\tilde a = a(1+\Order{\e^{-\theta/\eps}})$, which proves~\eqref{eq:bld16}
for $k=n$. 

The remaining eigenvalues $\lambda_2,\dots,\lambda_{n-1}$ are those of
$T^{11}$. Adding, if necessary, a cemetery state, we can make
$T^{11}$ a generator (meaning that we add an identically zero first row to
$T^{11}$ and a first column such that the row sums are all zero).
Corollary~\ref{cor_b1d2} shows that $T^{11}$ admits the same metastable
hierarchy as $L$, up to negligible error terms. Thus the result follows by
induction on the size of $L$.  
\end{proof}

We have thus proved Relation~\eqref{eq:asym05} in Theorem~\ref{thm_asym}, and
by extension the corresponding statements in Theorem~\ref{thm_dim1_trivial} and
Theorem~\ref{thm_dim1_nontrivial}. 

\subsection{The higher-dimensional case}

We consider now the case of an irreducible representation of dimension
$d\geqs2$. Then the generator $L$ has a block structure, with blocks whose
dimensions are multiples of $d$. We add a cemetery state to the system in such
a way that the row sums of $L$ vanish. 
We associate with $L$ an auxiliary matrix $L_*$ which has only one element
$\e^{-h^*(A_i,A_j)/\eps}$ for each pair $(i,j)$ of active orbits, plus the
cemetery state. 

Applying to $L$ the triangularisation algorithm described in
Section~\ref{ssec_triang} changes the blocks of $L$ to leading order according
to 
\begin{equation}
 \label{eq:bdd01}
 L_{ij} \mapsto 
 \widetilde L_{ij} 
 = L_{ij} - L_{in}L_{nn}^{-1}L_{nj}\;.
\end{equation} 
The algorithm induces a transformation on $L_*$ which is equivalent to the
one-dimensional algorithm discussed in the previous section. Thus we conclude
from Theorem~\ref{thm_triang_d1} that communication heights of $L_*$ are
preserved. 

Let us examine the following two cases.

\begin{itemiz}
\item 	Assume $j=s(i)$ is the successor of $i$. Then 
\begin{itemiz}
\item 	If $n \neq s(i)$, then $\widetilde L_{ij} =
L_{ij}[1+\Order{\e^{-\theta/\eps}}]$, because $L_{nj}$ is at most of order 
$L_{nn}$, and $L_{in}$ is negligible with respect to $L_{ij}$.
\item 	If $n = s(i)$, then either $j\neq s(n)$, and then again $\widetilde
L_{ij} = L_{ij}[1+\Order{\e^{-\theta/\eps}}]$, because $L_{nj}$ is negligible
with respect to $L_{nn}$. Or $j=s(n)$, and then $L_{in}L_{nn}^{-1}L_{nj}$ is
comparable to $L_{ij}$. 
\end{itemiz}
We thus conclude that $\widetilde L_{ij} =
L_{ij}[1+\Order{\e^{-\theta/\eps}}]$, unless the graph of successors
contains a path $i\to n\to j$, in which case the leading term of
$L_{ij}$ is modified according to~\eqref{eq:bdd01}. 

\item 	Consider now the case $j=i$. By the previous point, $L_{ii}$ is
modified to leading order by the triangularisation algorithm if and only if
the graph of successors contains a cycle  
$i\to n\to i$. Note that in this case, the modification involves the
two matrices $L_{in}$ and $L_{ni}$. These matrices cannot have been modified to
leading order at a previous step. Indeed, $L_{in}$ has been modified if and
only if there exists a $m\succ n$ such that the graph of successors contains a
path $i\to m\to n$. Assumption~\ref{assump_sym_nondeg} implies that this is
incompatible with the fact that the graph contains $i\to n\to i$. A similar
argument applies to $L_{ni}$. 
\end{itemiz}

It follows that at each step of the triangularisation algorithm, the diagonal
blocks $L_{ii}$ are preserved to leading order, unless $i$ is at the bottom of
a cycle in the graph of successors. This proves Theorem~\ref{thm_dimd}. 

%%%%%%%%%%%%%%%%%%%%%%%%%%%%%%%%%%%%%%%%%%%%%%%%%%%%%%%%%%%%%%%%%%%%%%%%%%%%%%

\section{Proofs -- Expected first-hitting times}
\label{sec_hit}

Fix a subset $A\subset\cX$, and let $w_A(x) = \expecin{x}{\tau_A}$ be the
first-hitting time of $A$ when starting in a point $x\in B=\cX\setminus A$. 
It is well known (see for instance~\cite[Chapter~3]{Norris_book}) that for any
$x\in B$, 
\begin{equation}
 \label{FK06}
 \sum_{y\in B} L_{xy}w_A(y) = -1\;.
\end{equation} 
If write $L$ as 
\begin{equation}
 \label{FK07}
 L = 
\begin{pmatrix}
L_{AA} & L_{AB} \\ L_{BA} & L_{BB}      
\end{pmatrix}\;,
\end{equation} 
then~\eqref{FK06} reads 
\begin{equation}
 \label{FK08}
 w_A = - L_{BB}^{-1}\vone\;.
\end{equation} 

\begin{prop}[Expected first-hitting time]
If $L$ satisfies the assumptions of Theorem~\ref{thm_asym} and
$A=\cM_k=\set{1,\dots k}$ with $k\geqs1$, then 
\begin{equation}
 \label{FK09}
 \expecin{x}{\tau_A}
 = \frac{1}{\abs{\lambda_{k+1}}} \brak{1+\Order{\e^{-\theta/\eps}}}
\end{equation}
holds for all $x\in A^c$. 
\end{prop}
\begin{proof}
The proof is by induction on the size $m=n-k$ of $L_{BB}$. The result is
obviously true if $m=1$, since the lower-right matrix element of $L$ is equal
to the eigenvalue $\lambda_n$, up to an error $1+\Order{\e^{-\theta/\eps}}$. 
Thus assume $m>1$ and write 
\begin{equation}
 \label{FK10}
 L_{BB} = 
 \begin{pmatrix}
 L^{11} & L^{12} \\ L^{21} & a
 \end{pmatrix}\;, 
\end{equation} 
with blocks $L^{11}\in\R^{(m-1)\times(m-1)}$,
$L^{12}\in\R^{(m-1)\times1}$, $L^{21}\in\R^{1\times(m-1)}$ and $a\in\R$. 
Using~\eqref{eq:bti03} and~\eqref{eq:bti04}, we see that 
\begin{align}
\nonumber
L_{BB}^{-1} \vone = ST^{-1}S^{-1} 
&= 
\begin{pmatrix}
\one & S_{12} \\ 0 & 1 
\end{pmatrix}
\begin{pmatrix}
(T^{11})^{-1} & 0 \\
-\tilde a^{-1}L^{21}(T^{11})^{-1} & \tilde a^{-1} 
\end{pmatrix}
\begin{pmatrix}
\one & -S_{12} \\ 0 & 1 
\end{pmatrix}
\begin{pmatrix}
\vone \\
1
\end{pmatrix}
\\
\nonumber
&= 
\begin{pmatrix}
\one & S_{12} \\ 0 & 1 
\end{pmatrix}
\begin{pmatrix}
(T^{11})^{-1} & 0 \\
-\tilde a^{-1}L^{21}(T^{11})^{-1} & \tilde a^{-1} 
\end{pmatrix}
\begin{pmatrix}
\vone - S^{12} \\
1
\end{pmatrix}
\\
&=
\begin{pmatrix}
\one & S_{12} \\ 0 & 1 
\end{pmatrix}
\begin{pmatrix}
(T^{11})^{-1} [\vone + \Order{\e^{-\theta/\eps}}] \\
\tilde a^{-1} \bigbrak{1 -
L^{21}(T^{11})^{-1}\vone(1+\Order{\e^{-\theta/\eps}})}
\end{pmatrix}\;.
 \label{FK11}
\end{align}
By induction, we may assume that 
\begin{equation}
 \label{FK12}
 (T^{11})^{-1}\vone = \frac{1}{\abs{\lambda_{k+1}}} \vone 
 \bigbrak{1 + \Order{\e^{-\theta/\eps}}}\;,
\end{equation} 
which implies 
\begin{equation}
 \label{FK13}
 L^{21}(T^{11})^{-1}\vone = \frac{L^{21}\vone}{\abs{\lambda_{k+1}}} 
 \bigbrak{1 + \Order{\e^{-\theta/\eps}}}
 = \frac{\abs{\lambda_n}}{\abs{\lambda_{k+1}}} 
 \bigbrak{1 + \Order{\e^{-\theta/\eps}}}\,.
\end{equation} 
Plugging this into~\eqref{FK11} and using the fact that
$\abs{\lambda_{k+1}}/\abs{\lambda_n} = \Order{\e^{-\theta/\eps}}$ we obtain 
\begin{equation}
 \label{FK14}
 L_{BB}^{-1} \vone 
= 
\begin{pmatrix}
\one & S_{12} \\ 0 & 1 
\end{pmatrix}
\begin{pmatrix}
\abs{\lambda_{k+1}}^{-1} \vone \bigbrak{1 + \Order{\e^{-\theta/\eps}}}  \\
\abs{\lambda_{k+1}}^{-1} \bigbrak{1 + \Order{\e^{-\theta/\eps}}} 
\end{pmatrix}
= \frac{1}{\abs{\lambda_{k+1}}}
\vone \bigbrak{1 + \Order{\e^{-\theta/\eps}}}\;,
\end{equation} 
which concludes the proof. 
\end{proof}

We point out that~\eqref{FK09} does not hold, in general, when $A$ is not a set
$\cM_k$ of the metastable hierarchy. For a counterexample, for instance
\cite[Example~3.2]{Berglund_irs_MPRF}.

This completes the proof of Theorem~\ref{thm_asym}, and thus also of
Theorem~\ref{thm_dim1_trivial} and Theorem~\ref{thm_dim1_nontrivial}.

\appendix

\section{Critical points of the constrained system}
\label{appendix} 

We give in this appendix a brief description of how we obtained the local
minima and saddles of index $1$ for Example~\ref{ex:Kawasaki} (a more detailed
analysis of this system will be published elsewhere). The case $N=4$ was first
studied in~\cite{Hun_Masterthesis}.

We look for extrema of the potential 
\begin{equation}
 \label{app01} 
 V_\gamma(x) = \sum_{i\in\Z/N\Z} U(x_i) + \frac{\gamma}{4}
\sum_{i\in\Z/N\Z}(x_{i+1}-x_i)^2
\end{equation} 
where $U(x)=\frac14 x^4 - \frac12 x^2$, under the constraint 
\begin{equation}
 \label{app02}
 \sum_{i\in\Z/N\Z} x_i = 0\;. 
\end{equation} 
We will apply a perturbative argument in $\gamma$, and thus start by
considering the case \mbox{$\gamma=0$}. Then the extremalisation problem is
equivalent to solving $\nabla V_0(x)=\lambda\vone$ on $\set{\sum x_i=0}$, or
equivalently 
\begin{equation}
\label{app03} 
f(x_i) = \lambda \qquad \forall i\in\Z/N\Z\;, 
\end{equation}
where $f(x)=U'(x)=x^3-x$ and $\lambda$ is the Lagrange multiplier. There are
three cases to consider:

\begin{enum}
\item 	If $\abs{\lambda}>2/(3\sqrt{3})$, then $f(x)=\lambda$ admits only one
real solution, different from zero, and the constrained problem has no
solution. 

\item 	If $\abs{\lambda}=2/(3\sqrt{3})$, then $f(x)=\lambda$ admits two real 
solutions, given by $\pm 1/\sqrt{3}$ and $\mp 2/\sqrt{3}$. Then solutions exist
only if $N$ is a multiple of $3$ (and they may give rise to degenerate families
of stationary points).  

\item 	If $\abs{\lambda}<2/(3\sqrt{3})$, then $f(x)=\lambda$ admits three
different real solutions $\alpha_0$, $\alpha_1$ and $\alpha_2$. We denote by
$n_i\in\N_0$ the number of $x_i$ equal to $\alpha_i$, and reorder the $\alpha_i$
in such a way that $n_0\leqs n_1\leqs n_2$. Then the constrained problem is
equivalent to 
\begin{align}
\nonumber
\alpha_0 + \alpha_1 + \alpha_2 &= 0 \\
\nonumber
\alpha_0 \alpha_1 \alpha_2 &= \lambda \\
\nonumber
\alpha_0\alpha_1 + \alpha_0\alpha_2 + \alpha_1\alpha_2 &= -1 \\
n_0\alpha_0 + n_1\alpha_1 + n_2\alpha_2 &= 0
\label{app04} 
\end{align}
with $n_0+n_1+n_2=N$. 
This can be seen to be equivalent to 
\begin{equation}
 \label{app05}
 (\alpha_0,\alpha_1,\alpha_2) = 
 \pm\frac{1}{R^{1/2}}
 (n_2-n_1,n_0-n_2,n_1-n_0)
\end{equation} 
where $R=n_0^2+n_1^2+n_2^2-n_0n_1-n_0n_2-n_1n_2$. 
\end{enum}

This shows that if $N$ is not a multiple of $3$, then all solutions of the
constrained problem can be indexed by ordered triples $(n_0,n_1,n_2)$ of
non-negative integers whose sum is $N$. By examining the Hessian of the
potential (taking into account the constraint), one can prove the following
result.

\begin{theorem}
\label{thm:hessian} 
Assume $N\geqs5$ is not a multiple of $3$. Then for $\gamma=0$ 
\begin{enum}
\item 	all local minima are given by ordered triples $(0,n_1,N-n_1)$ with
$3n_1>N$; 
\item 	all saddles of index $1$ are given by ordered triples
$(1,n_1,N-n_1-1)$ with $3n_1>N$. 
\end{enum}
If $N=4$, then all local minima are given by the triple $(0,2,2)$ and all
saddles of index $1$ by the triple $(1,1,2)$. 
\end{theorem}

Using the De-Moivre--Laplace formula, one can show that for large $N$, the
number of local minima grows like $2^N$, while the number of saddles of index
$1$ grows like $N2^N$. 

In the case $N=4$, 
\begin{itemiz}
\item 	the triple $(0,2,2)$ yields $6$ local
minima, having each two coordinates equal to $1$ and two coordinates equal to
$-1$; 
\item 	the triple $(1,1,2)$ yields $12$ saddles of index $1$, having each two
coordinates equal to $0$, one coordinate equal to $1$ and the other one equal
to $-1$. 
\end{itemiz}
One can check the octahedral structure of the associated graph by
constructing paths from each saddle to two different local minima, along which
the potential decreases. For instance, the path $\setsuch{(1,t,-t,-1)}{-1\leqs
t\leqs 1}$ interpolates between the local minima $(1,-1,1,-1)$ and
$(1,1,-1,-1)$ via the saddle $(1,0,0,-1)$, and the value of the potential along
this path is $2U(t)$, which is decreasing in $\abs{t}$ on $[-1,1]$. 

\begin{table}[ht]
\begin{center}
\begin{tabular}{|c|c|c|c|c|}
\hline
\hlinespace
& $\gamma=0$ & $\gamma>0$ & & $V_\gamma$   \\
\hline 
\hline 
\hlinespace
$a$ & $(1,1,-1,-1)$ & $(x,x,-x,-x)$ & $x=\sqrt{1-\gamma}$ & $-(1-\gamma)^2$ \\
$b$ & $(1,-1,1,-1)$ & $(x,-x,x,-x)$ & $x=\sqrt{1-2\gamma}$ & $-(1-2\gamma)^2$ \\
$a$--$a'$ & $(1,0,-1,0)$ & $(x,0,-x,0)$ & $x=\sqrt{1-\gamma}$ &
$-\frac12(1-\gamma)^2$ \\
$a$--$b$& $(1,-1,0,0)$ & $(x,-x,y,-y)$ &
$x,y=\frac{\sqrt{2-\gamma}\pm\sqrt{2-5\gamma}}{\sqrt8}$ & 
$-\frac{1}{8}(4-12\gamma+7\gamma^2)$ \\
\hline 
\end{tabular}
\end{center}
\caption[]{Local minima and saddles of index $1$ for the case $N=4$, with the
value of the potential. The relevant heights are 
$h_{ab} = V_\gamma(a$--$b) - V_\gamma(a)$, 
$h_{ba} = V_\gamma(a$--$b) - V_\gamma(b)$, and 
$h_{aa'} = V_\gamma(a$--$a') - V_\gamma(a)$.}
\label{table:orbits_4} 
\end{table}

In the case $N=8$, 
\begin{itemiz}
\item 	the triple $(0,4,4)$ yields $\binom84 = 70$ local minima,
having each four coordinates equal to $1$ and four coordinates equal to $-1$; 
\item 	the triple $(0,3,5)$ yields $2\binom83 = 112$ local minima, having three
coordinates equal to $\pm\alpha=\pm 5/\sqrt{19}$ and five coordinates equal to
$\pm\beta = \mp 3/\sqrt{19}$; 
\item 	and the triple $(1,3,4)$ yields $2\frac
{8!}{1!3!4!} = 560$ saddles of index $1$, with one coordinate equal to
$\mp1/\sqrt{7}$, three coordinates equal to $\pm3/\sqrt{7}$ and four coordinates
equal to $\mp2/\sqrt{7}$. 
\end{itemiz}
The connection rules stated in Section~\ref{ssec:ex_N8} can again be checked by
constructing paths along which the potential decreases.

Since the Hessian is nondegenerate at the stationary points listed by
Theorem~\ref{thm:hessian}, the implicit-function theorem applies, and shows that
these points persist, with the same stability, for sufficiently small positive
$\gamma$. In the case $N=4$, the coordinates can even be computed explicitly
(Table~\ref{table:orbits_4}), drawing on the fact that they keep the same
symmetry as for $\gamma=0$.

The value of the potential at the stationary points can then be computed,
exactly for $N=4$ and perturbatively to second order in $\gamma$ for $N=8$,
which allows to determine the metastable hierarchy.

%\vfill

%\newpage

{\small
\bibliography{BD}
\bibliographystyle{abbrv}               
}

\newpage

 \tableofcontents

\vfill

\bigskip\bigskip\noindent
{\small
%Nils Berglund, S\'ebastien Dutercq \\ 
Universit\'e d'Orl\'eans, Laboratoire {\sc Mapmo} \\
{\sc CNRS, UMR 7349} \\
F\'ed\'eration Denis Poisson, FR 2964 \\
B\^atiment de Math\'ematiques, B.P. 6759\\
45067~Orl\'eans Cedex 2, France \\
{\it E-mail addresses: }{\tt nils.berglund@univ-orleans.fr}, 
{\tt sebastien.dutercq@univ-orleans.fr}

%%%%%%%%%%%%%%%%%%%%%%%%%%%%%%%%%%%%%%%%%%%%%%%%%%%%%%%%%%%%%%%%%%%%%%%%%%%%%%

\end{document}

%% file: graph_example1.tex
\begin{tikzpicture}[<->,>=stealth',main node/.style={circle,minimum
size=0.25cm,fill=blue!20,draw},x=2.5cm,y=2.5cm]

\node[main node] (2) at (0.0,0.0) {$2$};
\node[main node] (3) at (1.0,0.0) {$3$};
\node[main node] (1) at (2.0,0.0) {$1$};

\draw[-,black,thick] (2) -- (3);
\draw[-,black,thick] (3) -- (1);

\end{tikzpicture}

%% file: potential_communication.tex
\begin{tikzpicture}[<->,>=stealth',main node/.style={circle,minimum
size=0.25cm,fill=blue!20,draw},x=1.2cm,y=0.8cm]

\draw[thick,-,smooth,domain=0.04:5.964,samples=75,/pgf/fpu,/pgf/fpu/output
format=fixed] plot (\x, {
(\x-1)*(\x-5)*(67*\x^4-789*\x^3+2996*\x^2-4044*\x+864)/720 });

\draw[blue,thick] (0.548,-1.61) -- node[right] {$h_{23}$} (0.548,2.16);
\draw[blue,thick] (1.75,0.943)  -- node[right] {$h_{32}$} (1.75,2.16);
\draw[blue,thick] (2.857,0.943) -- node[right] {$h_{31}$} (2.857,3.125);
\draw[blue,thick] (4.18,-1.61)  -- node[left] {$h_{231}$} (4.18,3.125);
\draw[blue,thick] (5.47,-2.2)  -- node[right] {$h_{13}$} (5.47,3.125);

\draw[-,blue,dashed] (0.548,2.16) -- (1.75,2.16);
\draw[-,blue,dashed] (0.548,-1.61) -- node[below] {$V_2$} (4.18,-1.61);
\draw[-,blue,dashed] (1.75,0.943) -- (2.857,0.943);
\draw[-,blue,dashed] (2.857,3.125) -- node[above] {$V_{(1,3)}$} (5.47,3.125);

\node[main node] at (0.548,-2.21) {$2$};
\node[main node] at (2.857,0.343) {$3$};
\node[main node] at (5.47,-2.848) {$1$};

\end{tikzpicture}

%% file: potential_disconnect_graph.tex
\begin{tikzpicture}[>=stealth',main node/.style={circle,minimum
size=0.25cm,fill=blue!20,draw},x=1.2cm,y=0.8cm]

\draw[thick,-,smooth,domain=0.04:5.964,samples=75,/pgf/fpu,/pgf/fpu/output
format=fixed] plot (\x, {
(\x-1)*(\x-5)*(67*\x^4-789*\x^3+2996*\x^2-4044*\x+864)/720 });

\draw[blue,semithick] (2.857,2.16) -- node[above] {$e^*(3)$} (0.548,2.16);
\draw[blue,semithick] (0.548,3.125) -- (5.452,3.125);
\node[blue] at (4.18,3.525) {$e^*(2)$};

\draw[red,ultra thick] (2.857,0.943) -- node[right] {$H(3,\set{1,2})$}
(2.857,2.16);
\draw[red,ultra thick] (0.548,-1.61) -- node[left] {$H(2,1)$} (0.548,3.125);
\draw[red,ultra thick] (5.47,-2.248) -- (5.47,4.6);

\node[main node] at (0.548,-2.21) {$2$};
\node[main node] at (2.857,0.343) {$3$};
\node[main node] at (5.47,-2.848) {$1$};

\end{tikzpicture}

%% file: graph_successors_example1.tex
\begin{tikzpicture}[->,>=stealth',main node/.style={circle,minimum
size=0.25cm,fill=blue!20,draw},x=2.5cm,y=2.5cm]

%\node at (-0.5,0.5) {{\bf (a)}};

\node[main node] (2) at (0.0,0.0) {$2$};
\node[main node] (3) at (1.0,0.0) {$3$};
\node[main node] (1) at (2.0,0.0) {$1$};

\draw[->,thick,shorten >=1pt] (1) -- node[above] {$s$} (3);
\draw[->,thick,shorten >=1pt] (3) edge[bend left] node[below] {$s$} (2);
\draw[->,thick,shorten >=1pt] (2) edge[bend left] node[above] {$s$} (3);

\node[main node] (22) at (0.0,-1.0) {$2$};
\node[circle,minimum size=0.25cm,black!50,draw] (32) at (1.0,-1.0)
{$3$};
\node[main node] (12) at (2.0,-1.0) {$1$};

\draw[->,thick,shorten >=1pt] (12) edge[bend left] node[below] {$s$} (22);
\draw[->,thick,shorten >=1pt] (22) edge[bend left] node[above] {$s$} (12);

\end{tikzpicture}

%% file: potential_fill_graph.tex
\begin{tikzpicture}[>=stealth',main node/.style={circle,minimum
size=0.25cm,fill=blue!20,draw},x=1.2cm,y=0.8cm]

\path[fill=black!20,-,smooth,domain=1.8:3.6,samples=25,/pgf/fpu,
/pgf/fpu/output format=fixed] 
plot (\x, {
(\x-1)*(\x-5)*(67*\x^4-789*\x^3+2996*\x^2-4044*\x+864)/720 }) --
(3.6,2.16) -- (1.8,2.16);

\draw[semithick] (1.8,2.16) -- (3.6,2.16);

\draw[thick,-,smooth,domain=0.04:5.964,samples=75,/pgf/fpu,/pgf/fpu/output
format=fixed] plot (\x, {
(\x-1)*(\x-5)*(67*\x^4-789*\x^3+2996*\x^2-4044*\x+864)/720 });

\node[main node] at (0.548,-2.21) {$2$};
\node[circle,minimum size=0.25cm,black!50,draw] at (2.857,0.343) {$3$};
\node[main node] at (5.47,-2.848) {$1$};

\end{tikzpicture}

%% file: graph_algorithm_example1.tex
\begin{tikzpicture}[->,>=stealth',main node/.style={circle,minimum
size=0.25cm,fill=blue!20,draw},x=2.5cm,y=2.5cm]

%\node at (-0.5,0.5) {{\bf (b)}};

\node[main node] (2) at (0.0,0.0) {$2$};
\node[main node] (3) at (1.0,0.0) {$3$};
\node[main node] (1) at (2.0,0.0) {$1$};

\draw[->,thick,shorten >=1pt] (1) edge[bend left] node[below] {$h_{13}$} (3);
\draw[->,thick,shorten >=1pt] (3) edge[bend left] node[above] {$h_{31}$} (1);
\draw[->,thick,shorten >=1pt] (3) edge[bend left] node[below] {$h_{32}$} (2);
\draw[->,thick,shorten >=1pt] (2) edge[bend left] node[above] {$h_{23}$} (3);

\node[main node] (21) at (0.0,-1.0) {$2$};
\node[main node] (31) at (1.0,-1.0) {$3$};
\node[main node] (11) at (2.0,-1.0) {$1$};

\draw[->,thick,shorten >=1pt] (11) edge[bend left] node[below] {$h_{13}$} (31);
\draw[->,thick,shorten >=1pt] (31) edge[bend left] node[above] 
{$h_{31}-h_{32}$} (11);
\draw[->,thick,shorten >=1pt] (31) edge[bend left] node[below] {$0$} (21);
\draw[->,thick,shorten >=1pt] (21) edge[bend left] node[above] {$h_{23}$} (31);

\node[main node] (22) at (0.0,-2.0) {$2$};
\node[circle,minimum size=0.25cm,black!50,draw] (32) at (1.0,-2.0)
{$3$};
\node[main node] (12) at (2.0,-2.0) {$1$};

\draw[->,thick,shorten >=1pt] (12) edge[bend left] node[below] {$h_{13}$} (22);
\draw[->,thick,shorten >=1pt] (22) edge[bend left] node[above] {$h_{231}$} (12);

\end{tikzpicture}

%% file: quintuple_well.tex
\begin{tikzpicture}[>=stealth',scale=0.7]

% grid to help positioning
%\draw[help lines] (-3,-1) grid (7,4);

% potential

\draw[black,thick] plot[smooth,tension=.55]
  coordinates{(-2.5,4.5) (-1,0) (-0.1,1.75) (0.5,1) (1.15,3) (1.75,2) (2.5,4)
  (3.5,1) (4,1.25) (5.1,-1) (7,4.5)};

% minima

\node[] (1) at (5.0,-1) {}; 
\node[blue] [below of=1,yshift=0.6cm] {\small $1$};

\node[] (2) at (-0.9,0) {}; 
\node[blue] [below of=2,yshift=0.6cm] {\small $2$};

\node[] (3) at (1.75,2) {}; 
\node[blue] [below of=3,yshift=0.65cm] {\small $3$};

\node[] (4) at (0.5,1) {}; 
\node[blue] [below of=4,yshift=0.65cm] {\small $4$};

\node[] (5) at (3.55,1) {}; 
\node[blue] [below of=5,yshift=0.6cm] {\small $5$};

\end{tikzpicture}

%% file: graphe_4.tex
\begin{tikzpicture}

%\node [anchor=north west] (1) at (3*sin{\angxy},3*cos{\angxy}*sin{\angxz}) {\footnotesize{(\sitep,\sitem,\sitem,\sitep)}};
\pos{(3*sin{\angxy}+0.5}{3*cos{\angxy}*sin{\angxz}-0.5};
\carre{\sitep}{\sitem}{\sitem}{\sitep};

%\node [anchor=west] (2) at (3*cos{\angxy},-3*sin{\angxy}*sin{\angxz}) {\footnotesize{(\sitep,\sitep,\sitem,\sitem)}};
\pos{3*cos{\angxy}+0.5}{-3*sin{\angxy}*sin{\angxz}};
\carre{\sitep}{\sitep}{\sitem}{\sitem};

%\node [anchor= south west](3) at (-3*sin{\angxy},-3*cos{\angxy}*sin{\angxz}) {\footnotesize{(\sitem,\sitep,\sitep,\sitem)}};
\pos{-3*sin{\angxy}-0.5}{-3*cos{\angxy}*sin{\angxz}+0.5};
\carre{\sitem}{\sitep}{\sitep}{\sitem};

%\node [anchor=east] (4) at (-3*cos{\angxy},3*sin{\angxy}*sin{\angxz}) {\footnotesize{(\sitem,\sitem,\sitep,\sitep)}};
\pos{-3*cos{\angxy}-0.5}{3*sin{\angxy}*sin{\angxz}};
\carre{\sitem}{\sitem}{\sitep}{\sitep};

%\node [anchor=south] (5) at (0,3*cos{\angxz}) {\footnotesize{(\sitep,\sitem,\sitep,\sitem)}};
\pos{0}{3*cos{\angxz}+0.5};
\carre{\sitep}{\sitem}{\sitep}{\sitem};

%\node [anchor=north] (6) at (0,-3*cos{\angxz}) {\footnotesize{(\sitem,\sitep,\sitem,\sitep)}};
\pos{0}{-3*cos{\angxz}-0.5};
\carre{\sitem}{\sitep}{\sitem}{\sitep};

\draw[thick] (-3*cos{\angxy},3*sin{\angxy}*sin{\angxz}) -- (3*sin{\angxy},3*cos{\angxy}*sin{\angxz}) -- (3*cos{\angxy},-3*sin{\angxy}*sin{\angxz});
\draw[thick,dash pattern=on 2mm off 1mm] (3*cos{\angxy},-3*sin{\angxy}*sin{\angxz}) -- (-3*sin{\angxy},-3*cos{\angxy}*sin{\angxz}) -- (-3*cos{\angxy},3*sin{\angxy}*sin{\angxz});
\draw[thick] (0,3*cos{\angxz}) -- (3*sin{\angxy},3*cos{\angxy}*sin{\angxz}) -- (0,-3*cos{\angxz});
\draw[thick] (0,3*cos{\angxz}) -- (3*cos{\angxy},-3*sin{\angxy}*sin{\angxz}) -- (0,-3*cos{\angxz});
\draw[thick,dash pattern=on 2mm off 1mm] (0,3*cos{\angxz}) -- (-3*sin{\angxy},-3*cos{\angxy}*sin{\angxz}) -- (0,-3*cos{\angxz});
\draw[thick] (0,3*cos{\angxz}) -- (-3*cos{\angxy},3*sin{\angxy}*sin{\angxz}) -- (0,-3*cos{\angxz});

\end{tikzpicture}

%% file: graphe_4_orbit.tex
\begin{tikzpicture}
\pos{0}{-1}
\carre{\sitep}{\sitep}{\sitem}{\sitem}
\node[shape=circle, minimum size=1cm] (A1) at (0,-1) {$A_1$};

\pos{2}{-1}
\carre{\sitep}{\sitem}{\sitep}{\sitem}
\node[shape=circle, minimum size=1cm] (A2) at (2,-1) {$A_2$};

\draw[semithick]    (A1) -- (A2);

\end{tikzpicture}

%% file: graphe_8_labels_3.tex
\begin{tikzpicture}

%%%%%%%%%%%%%%%%%%%%%%%%%
%		Orbites			%
%%%%%%%%%%%%%%%%%%%%%%%%%

\pos{0}{0}
\octogon{\sitep}{\sitep}{\sitep}{\sitep}{\sitem}{\sitem}{\sitem}{\sitem}
\node[shape=circle, minimum size=1.2cm] (A1) at (0,0) {$A_1$};

\pos{2.5}{0}
\octogon{\sitea}{\sitea}{\sitea}{\siteb}{\siteb}{\siteb}{\siteb}{\siteb}
\node[shape=circle, minimum size=1.2cm] (A9) at (2.5,0) {$A_9$};

\pos{5}{0}
\octogon{\sitep}{\sitep}{\sitep}{\sitem}{\sitep}{\sitem}{\sitem}{\sitem}
\node[shape=circle, minimum size=1.2cm] (A4) at (5,0) {$A_4$};

\pos{10.0}{-2.5}
\octogon{\sitea}{\siteb}{\sitea}{\siteb}{\sitea}{\siteb}{\siteb}{\siteb}
\node[shape=circle, minimum size=1.2cm] (A11) at (10,-2.5) {$A_{11}$};

\pos{12.5}{-2.5}
\octogon{\sitep}{\sitem}{\sitep}{\sitem}{\sitep}{\sitem}{\sitep}{\sitem}
\node[shape=circle, minimum size=1.2cm] (A7) at (12.5,-2.5) {$A_7$};

\pos{0}{-2.5}
\octogon{\sitea}{\sitea}{\siteb}{\sitea}{\siteb}{\siteb}{\siteb}{\siteb}
\node[shape=circle, minimum size=1.2cm] (A12) at (0,-2.5) {$A_{12}$};

\pos{2.5}{-2.5}
\octogon{\sitep}{\sitep}{\sitep}{\sitem}{\sitem}{\sitep}{\sitem}{\sitem}
\node[shape=circle, minimum size=1.2cm] (A3) at (2.5,-2.5) {$A_3$};

\pos{5}{-2.5}
\octogon{\sitea}{\sitea}{\siteb}{\siteb}{\sitea}{\siteb}{\siteb}{\siteb}
\node[shape=circle, minimum size=1.2cm] (A8) at (5,-2.5) {$A_8$};

\pos{7.5}{-2.5}
\octogon{\sitep}{\sitep}{\sitem}{\sitem}{\sitep}{\sitep}{\sitem}{\sitem}
\node[shape=circle, minimum size=1.2cm] (A2) at (7.5,-2.5) {$A_2$};

\pos{0}{-5}
\octogon{\sitep}{\sitep}{\sitem}{\sitep}{\sitem}{\sitem}{\sitep}{\sitem}
\node[shape=circle, minimum size=1.2cm] (A5) at (0,-5) {$A_5$};

\pos{2.5}{-5}
\octogon{\sitea}{\siteb}{\sitea}{\siteb}{\siteb}{\sitea}{\siteb}{\siteb}
\node[shape=circle, minimum size=1.2cm] (A10) at (2.5,-5) {$A_{10}$};

\pos{5}{-5}
\octogon{\sitep}{\sitep}{\sitem}{\sitep}{\sitem}{\sitep}{\sitem}{\sitem}
\node[shape=circle, minimum size=1.2cm] (A6) at (5,-5) {$A_6$};

%%%%%%%%%%%%%%%%%%%%%%%%%
%		Liens			%
%%%%%%%%%%%%%%%%%%%%%%%%%

\draw[thick]    (A1) -- node[above] {$32$} (A9);

\draw[thick]    (A9) -- node[above] {$32$} (A4);

\draw[thick]    (A11) -- node[above] {$16$} (A7);

\draw[thick]    (A1) -- node[right] {$32$} (A12);

\draw[thick]    (A9) -- node[right] {$16$} (A3);

\draw[thick]    (A4) -- node[right] {$32$} (A8);

\draw[thick]    (A12) -- node[above] {$32$} (A3);
    
\draw[thick]    (A3) -- node[above] {$64$} (A8);

\draw[thick]    (A8) -- node[above] {$32$} (A2);

\draw[thick]    (A12) -- node[right] {$32$} (A5);

\draw[thick]    (A5) -- node[above] {$32$} (A10);

\draw[thick]    (A6) -- node[above] {$32$} (A10);

\draw[thick]    (A3) -- node[right] {$16$} (A10);

\draw[thick] (A6) -- node[right] {$32$} (A8);

\draw[thick] (A4) to[out=-10,in=135] node[above] {$32$} (A11);

\draw[thick] (A6) to[out=10,in=-135] node[below] {$32$} (A11);

\draw[thick] (A12) to[out=140,in=120,distance=4cm] node[above] {$32$} (A4);

\draw[thick] (A12) to[out=220,in=-120,distance=4cm] node[below] {$32$} (A6);

\end{tikzpicture}

%% file: graphe_8_ex1_3.tex
\begin{tikzpicture}
            
\pos{0}{0}
\octogon{\sitep}{\sitep}{\sitep}{\sitep}{\sitem}{\sitem}{\sitem}{\sitem}
\node[shape=circle, minimum size=1.2cm] (A1) at (0,0) {$A_1$};

\pos{2.5}{0}
%\octogon{\sitea}{\sitea}{\sitea}{\siteb}{\siteb}{\siteb}{\siteb}{\siteb}
\cimetiere
\node[shape=circle, minimum size=1.2cm] (A9) at (2.5,0) {$ $};

\pos{5}{0}
\octogon{\sitep}{\sitep}{\sitep}{\sitem}{\sitep}{\sitem}{\sitem}{\sitem}
\node[shape=circle, minimum size=1.2cm] (A4) at (5,0) {$A_4$};

\pos{10}{-2.5}
%\octogon{\sitea}{\siteb}{\sitea}{\siteb}{\sitea}{\siteb}{\siteb}{\siteb}
\cimetiere
\node[shape=circle, minimum size=1.2cm] (A11) at (10,-2.5) {$ $};

\pos{12.5}{-2.5}
%\octogon{\sitep}{\sitem}{\sitep}{\sitem}{\sitep}{\sitem}{\sitep}{\sitem}
\cimetiere
\node[shape=circle, minimum size=1.2cm] (A7) at (12.5,-2.5) {$ $};

\pos{0}{-2.5}
\octogon{\sitea}{\sitea}{\siteb}{\sitea}{\siteb}{\siteb}{\siteb}{\siteb}
\node[shape=circle, minimum size=1.2cm] (A12) at (0,-2.5) {$A_{12}$};

\pos{2.5}{-2.5}
%\octogon{\sitep}{\sitep}{\sitep}{\sitem}{\sitem}{\sitep}{\sitem}{\sitem}
\cimetiere
\node[shape=circle, minimum size=1.2cm] (A3) at (2.5,-2.5) {$ $};

\pos{5}{-2.5}
\octogon{\sitea}{\sitea}{\siteb}{\siteb}{\sitea}{\siteb}{\siteb}{\siteb}
\node[shape=circle, minimum size=1.2cm] (A8) at (5,-2.5) {$A_8$};

\pos{7.5}{-2.5}
\octogon{\sitep}{\sitep}{\sitem}{\sitem}{\sitep}{\sitep}{\sitem}{\sitem}
\node[shape=circle, minimum size=1.2cm] (A2) at (7.5,-2.5) {$A_2$};

% \pos{0}{-5}
% \octogon{\sitep}{\sitep}{\sitem}{\sitep}{\sitem}{\sitem}{\sitep}{\sitem}
% \node[shape=circle, minimum size=1.2cm] (A5) at (1.5,-5) {$A_5$};
% 
% \pos{2.5}{-5}
% %\octogon{\sitea}{\siteb}{\sitea}{\siteb}{\siteb}{\sitea}{\siteb}{\siteb}
% \cimetiere
% \node[shape=circle, minimum size=1.2cm] (A10) at (4,-5) {$ $};
% 
% \pos{5}{-5}
% \octogon{\sitep}{\sitep}{\sitem}{\sitep}{\sitem}{\sitep}{\sitem}{\sitem}
% \node[shape=circle, minimum size=1.2cm] (A6) at (5,-7.5) {$A_6$};

\pos{0}{-5}
\octogon{\sitep}{\sitep}{\sitem}{\sitep}{\sitem}{\sitem}{\sitep}{\sitem}
\node[shape=circle, minimum size=1.2cm] (A5) at (0,-5) {$A_5$};

\pos{2.5}{-5}
%\octogon{\sitea}{\siteb}{\sitea}{\siteb}{\siteb}{\sitea}{\siteb}{\siteb}
\cimetiere
\node[shape=circle, minimum size=1.2cm] (A10) at (2.5,-5) {$ $};

\pos{5}{-5}
\octogon{\sitep}{\sitep}{\sitem}{\sitep}{\sitem}{\sitep}{\sitem}{\sitem}
\node[shape=circle, minimum size=1.2cm] (A6) at (5,-5) {$A_6$};

\draw[semithick]    (A1) -- (A9);

\draw[semithick]    (A9) -- (A4);

\draw[semithick]    (A11) -- (A7);

\draw[semithick]    (A1) -- (A12);

\draw[semithick]    (A9) -- (A3);

\draw[semithick]    (A4) -- (A8);

\draw[semithick]    (A12) -- (A3);
    
\draw[semithick]    (A3) -- (A8);

\draw[semithick]    (A8) -- (A2);

\draw[semithick]    (A12) -- (A5);

\draw[semithick]    (A6) -- (A10);

\draw[semithick]    (A5) -- (A10);

\draw[semithick]    (A3) -- (A10);

\draw[semithick]    (A6) to (A8);

\draw[semithick]    (A4) to[out=-10,in=135] (A11);

\draw[semithick]    (A6) to[out=10,in=-135] (A11);

\draw[semithick]    (A12) to[out=140,in=120,distance=4cm] (A4);

\draw[semithick]    (A12) to[out=220,in=-120,distance=4cm] (A6);

\end{tikzpicture}

%% file: graphe_8_successeur.tex
\begin{tikzpicture}[>=stealth']

%%%%%%%%%%%%%%%%%%%%%%%%%
%		Orbites			%
%%%%%%%%%%%%%%%%%%%%%%%%%

\pos{0}{0}
\octogon{\sitep}{\sitep}{\sitep}{\sitep}{\sitem}{\sitem}{\sitem}{\sitem}
\node[shape=circle, minimum size=1.2cm] (A1) at (0,0) {$A_1$};

\pos{2.5}{0}
\octogon{\sitea}{\sitea}{\sitea}{\siteb}{\siteb}{\siteb}{\siteb}{\siteb}
\node[shape=circle, minimum size=1.2cm] (A9) at (2.5,0) {$A_9$};

\pos{5}{0}
\octogon{\sitep}{\sitep}{\sitep}{\sitem}{\sitep}{\sitem}{\sitem}{\sitem}
\node[shape=circle, minimum size=1.2cm] (A4) at (5,0) {$A_4$};

\pos{10.0}{-2.5}
\octogon{\sitea}{\siteb}{\sitea}{\siteb}{\sitea}{\siteb}{\siteb}{\siteb}
\node[shape=circle, minimum size=1.2cm] (A11) at (10,-2.5) {$A_{11}$};

\pos{12.5}{-2.5}
\octogon{\sitep}{\sitem}{\sitep}{\sitem}{\sitep}{\sitem}{\sitep}{\sitem}
\node[shape=circle, minimum size=1.2cm] (A7) at (12.5,-2.5) {$A_7$};

\pos{0}{-2.5}
\octogon{\sitea}{\sitea}{\siteb}{\sitea}{\siteb}{\siteb}{\siteb}{\siteb}
\node[shape=circle, minimum size=1.2cm] (A12) at (0,-2.5) {$A_{12}$};

\pos{2.5}{-2.5}
\octogon{\sitep}{\sitep}{\sitep}{\sitem}{\sitem}{\sitep}{\sitem}{\sitem}
\node[shape=circle, minimum size=1.2cm] (A3) at (2.5,-2.5) {$A_3$};

\pos{5}{-2.5}
\octogon{\sitea}{\sitea}{\siteb}{\siteb}{\sitea}{\siteb}{\siteb}{\siteb}
\node[shape=circle, minimum size=1.2cm] (A8) at (5,-2.5) {$A_8$};

\pos{7.5}{-2.5}
\octogon{\sitep}{\sitep}{\sitem}{\sitem}{\sitep}{\sitep}{\sitem}{\sitem}
\node[shape=circle, minimum size=1.2cm] (A2) at (7.5,-2.5) {$A_2$};

\pos{0}{-5}
\octogon{\sitep}{\sitep}{\sitem}{\sitep}{\sitem}{\sitem}{\sitep}{\sitem}
\node[shape=circle, minimum size=1.2cm] (A5) at (0,-5) {$A_5$};

\pos{2.5}{-5}
\octogon{\sitea}{\siteb}{\sitea}{\siteb}{\siteb}{\sitea}{\siteb}{\siteb}
\node[shape=circle, minimum size=1.2cm] (A10) at (2.5,-5) {$A_{10}$};

\pos{5}{-5}
\octogon{\sitep}{\sitep}{\sitem}{\sitep}{\sitem}{\sitep}{\sitem}{\sitem}
\node[shape=circle, minimum size=1.2cm] (A6) at (5,-5) {$A_6$};

%%%%%%%%%%%%%%%%%%%%%%%%%
%		Liens			%
%%%%%%%%%%%%%%%%%%%%%%%%%

\draw[->,shorten >=1pt] (A1) edge[bend left]  (A9);

\draw[->,shorten >=1pt] (A9) edge[bend left] (A1);

\draw[->,shorten >=1pt]    (A12) -- (A1);

\draw[->,shorten >=1pt]    (A5) -- (A12);

\draw[->,shorten >=1pt]    (A3) -- (A9);

\draw[->,shorten >=1pt]    (A10) -- (A3);

\draw[->,shorten >=1pt]    (A4) -- (A9);

\draw[->,shorten >=1pt]    (A8) -- (A4);

\draw[->,shorten >=1pt]    (A2) -- (A8);

\draw[->,shorten >=1pt]    (A7) -- (A11);

\draw[->,shorten >=1pt] (A11) to[out=135,in=-10]  (A4);

\draw[->,shorten >=1pt] (A6) to[out=-120,in=220,distance=4cm] (A12);

\end{tikzpicture}